\documentclass[english,openary,11pt]{article}
\usepackage{cmap}
\usepackage[english]{babel}
\usepackage{graphicx}
\DeclareGraphicsExtensions{.png}
\usepackage{floatflt}
\usepackage{amsfonts}                              
\usepackage{amsthm}                          
\usepackage{amsmath}
\usepackage{mathtools}                 
\usepackage{gensymb}
\usepackage{amssymb,amsfonts}
\usepackage[usenames]{color}
\usepackage{wrapfig}
\usepackage{url}
\usepackage{cmap}
\usepackage{graphicx}
\usepackage{multirow}
\usepackage{comment}
\usepackage{caption}
\usepackage{subcaption}
\usepackage{float}
\usepackage{diagbox}
\usepackage{enumerate}
\usepackage{titling}
\usepackage{makecell}
\usepackage{booktabs}
\usepackage{amsthm}
\usepackage{bm}
\usepackage{subcaption}

\usepackage[hidelinks,hyperindex,breaklinks]{hyperref}
\usepackage{enumitem}
\usepackage[nottoc,numbib]{tocbibind} 
\setlist[itemize]{noitemsep, topsep=3pt, leftmargin=18pt}
\setlist[enumerate]{noitemsep, topsep=3pt, leftmargin=22pt}
\usepackage[title]{appendix}
\usepackage{breakcites}

\allowdisplaybreaks

\usepackage{tikz}
\usetikzlibrary{positioning, fit, arrows, arrows.meta,
calc, decorations.pathreplacing, backgrounds}
\usepackage{pgfplots}

\usepackage{soul}               
\setul{}{0.6pt}
\setstcolor{blue}

\makeatletter
\renewcommand\section{\@startsection {section}{1}{\z@}%
                                   {0ex \@plus 0ex \@minus 0ex}%
                                   {0.5ex \@plus0ex}%
                                   {\normalfont\LARGE\bfseries}}
\renewcommand\subsection{\@startsection{subsection}{2}{\z@}%
                                     {0ex\@plus 0ex \@minus 0ex}%
                                     {0.5ex \@plus 0ex}%
                                     {\normalfont\Large\bfseries}}
\renewcommand\subsubsection{\@startsection{subsubsection}{3}{\z@}%
                                     {0ex\@plus 0ex \@minus 0ex}%
                                     {0.5ex \@plus 0ex}%
                                     {\normalfont\large\bfseries}}
\makeatother

\makeatletter
\def\thm@space@setup{%
  \thm@preskip=2mm
  \thm@postskip=\thm@preskip 
}
\makeatother

\newcommand{\red}[1]{\textcolor{red}{#1}}

\usepackage[titles]{tocloft}
\newcommand\xqed[1]{%
  \leavevmode\unskip\penalty9999 \hbox{}\nobreak\hfill
  \quad\hbox{#1}}
\newcommand\demo{\xqed{$\triangle$}}
\newcommand{\demoo}{\tag*{$\triangle$}}

\usepackage[left=1.65cm,right=1.65cm,
    top=1.6cm,bottom=1.6cm,bindingoffset=0cm]{geometry}

\theoremstyle{definition}
\newtheorem{remark}{Remark}
\newtheorem{conjecture}{Conjecture}
\newtheorem{example}{Example}
\newtheorem{definition}{Definition}

\newtheorem{theorem}{Theorem}
\newtheorem{lemma}{Lemma}

\newtheorem{proposition}{Proposition}
\newtheorem{question}{Question}

\newtheorem*{remark-conjecture}{Remark-conjecture}

\newcommand{\R}{\mathbb{R}}
\newcommand{\N}{\mathbb{N}}

\begin{document}
\setlength{\abovedisplayskip}{5pt}
\setlength{\belowdisplayskip}{5pt}
\setlength\abovedisplayshortskip{5pt}
\setlength\belowdisplayshortskip{5pt}

\author{Ivan Novikov, Université Paris~1 Panthéon-Sorbonne}
\title{\vspace{-2.1cm} Asymptotic Value in Zero-Sum Stochastic Games with \\ Vanishing Stage Duration and Public Signals
\vspace{-0.42cm}}
\date{}

\maketitle

\vspace{-1.8cm}

\begin{abstract}

We study $\lambda$-discounted zero-sum games as the discount factor $\lambda$ approaches $0$ (that is, the players are more and more patient), in the context of games with stage duration. In stochastic games with stage duration $h$, players act at times $0, h, 2h$, and so on. The payoff and leaving probabilities are proportional to $h$. When $h$ tends to $0$, such discrete-time games approximate games played in continuous time. The asymptotic behavior of the values (when both $\lambda$ and $h$ tend to $0$) has already been studied for stochastic games with full state observation and for state-blind games.

We consider the same question for the case of stochastic games with deterministic public signals on the state. We construct a stochastic game with public signals, with no asymptotic value (as the discount factor $\lambda$ goes to $0$) if the stage duration is $1$, but with an asymptotic value when the stage duration $h$ and the discount factor $\lambda$ both tend to $0$. Informally, this means that the asymptotic value in discrete time does not exist, whereas it does exist in continuous time. This situation cannot occur in stochastic games with full state observation. 

\vspace{0.4cm}


\end{abstract}

\tableofcontents

\newpage


\textbf{Notation:}
\begin{itemize}
	\item $\N^*$ is the set of all positive integers; $\N_0 := \N^* \cup \{0\};$
	\item $\R_+ := \{x : x \in \R \text{ and } x \ge 0\}$;
	\item If $C$ is a finite set, then 
	$\Delta (C)$ is the set of probability measures on $C$;
  \item If $C$ is a set, then $\Delta_f(C)$ is the set of all probability  measures on $C$ with finite support;
	\item If $X$ is a finite set, and $f,g : X \to \R$ are two functions, then
	$\langle f (\cdot), g (\cdot)\rangle := \sum_{x\in X} f(x) g(x);$\\
  If $x = (x_1, \ldots, x_n), y = (y_1, \ldots, y_n) \in X$, then 
  $\langle x_1, x_2 \rangle := \sum_{i=1}^n x_i y_i$;
	\item If $I, J, \Omega$ are finite sets, $\zeta \in \Delta(\Omega)$, and 
	$\mu(\cdot|i, j, \omega') \in \Delta(\Omega)$ for any $i \in I, j \in J,\omega' \in \Omega$, then for each $i \in I, j \in J, \omega \in \Omega$ we define \vspace{-0.1cm}
	$$\left(\zeta * \mu(i,j)\right)(\omega) := 
	\sum\nolimits_{\omega'\in \Omega} \zeta(\omega') \cdot \mu(\omega|i,j,\omega');$$ \vspace{-0.5cm}
	\item If $I, J$ are finite sets and $g : I \times J \to \R$ is a function, then
	\begin{multline*}
  \texttt{Val}_{I \times J} [g(i,j)] := 
	\sup_{x \in \Delta(I)} \inf_{y \in \Delta(J)} 
	\left(\int_{I \times J} g (i, j) \; dx (i) \otimes dy (j)\right) =\\
	\inf_{y \in \Delta(J)} \sup_{x \in \Delta(I)} 
	\left(\int_{I \times J} g (i, j) \; dx (i) \otimes dy (j)\right),
  \end{multline*}
	i.e., $\texttt{Val}_{I \times J} [g(i,j)]$ is the value of the one-shot 
	zero-sum game with action spaces $I, J$ and with payoff function $g$.
\end{itemize}

\section{Introduction}

\textbf{Zero-sum stochastic games} with full state observation were first introduced in \cite{Sha53} to model dynamic interactions between two players with opposing interests. This game proceeds in discrete time as follows. 
At each stage, Player~1 and Player~2 observe the current state, recall the actions taken in previous stages, and simultaneously choose mixed actions. These choices determine the stage payoff, which Player~1 seeks to maximize and Player~2 seeks to minimize. The next state is then drawn randomly, conditional on the current state and actions.
Given a fixed discount factor $\lambda \in (0,1]$ and an initial state $\omega$, Player 1 aims to maximize the $\lambda$-discounted total payoff
$E\left[\lambda \sum\nolimits_{i = 1}^\infty (1-\lambda)^{i-1} g_i\right],$
where $g_i$ denotes the payoff at stage $i$, and Player 2 seeks to minimize it. Under standard measurability assumptions, the maxmin and minmax coincide, and the resulting quantity is called the $\lambda$-discounted value, denoted by $v_\lambda(\omega)$.

Stochastic games with public signals \cite[\S IV.1]{MerSorZam15} concern the case where players cannot fully observe the current state. In such games, the players are informed of an initial probability distribution $p$ over the states, and at each stage they observe the action profile together with a public signal that depends on the current state. The value $v_\lambda(p)$ is defined in the same way as above. A special case of such games is that of state-blind stochastic games, in which players observe only a single public signal.

One line of research concerns the behavior of the value as the players become increasingly patient, i.e., as $\lambda$ tends to $0$. For different families of stochastic games, $v_\lambda(\cdot)$ converges uniformly as $\lambda$ tends to $0$; see, e.g., \cite{Eve58}, \cite{Koh74}, \cite{BewKoh76}, \cite{Ros00}, \cite{RosSor01}, \cite{BolGauVig15}, \cite{Zil24}. However, $v_\lambda$ may diverge as $\lambda$ tends to $0$, as in finite state-blind stochastic games \cite{Zil16}, or in games with full state observation, compact action spaces, finite state space, and continuous payoff and transition functions \cite{Vig13}.
See also \cite{SorVig15} for a general method to construct counterexamples in various frameworks.

\textbf{Zero-sum stochastic games with full state observation and with stage duration} were introduced in \cite{Ney13}. Given a stochastic game $\Gamma$ (with full state observation),  \cite{Ney13} considers a family $\Gamma_h$ of stochastic games in which the players act at times $0, h, 2h, $ and so on. In such games, both payoffs and leaving probabilities\footnote{A \emph{leaving probability} refers to any transition probability between two \emph{distinct} states.} are normalized at each stage, i.e., they are proportional to $h$. Let $v_{h,\lambda}$ denote the value of the game with discount factor $\lambda$ and stage duration $h$.
We are interested in the following two limits. 
\begin{enumerate}
\item For each fixed $h \in (0,1]$, the \emph{asymptotic value} $\lim_{\lambda \to 0} v_{h, \lambda}$.
\item For each fixed $\lambda \in (0,1]$, the \emph{evanescent value} $\lim_{h \to 0} v_{h, \lambda}$.
\end{enumerate}
The evanescent value $\lim_{h \to 0} v_{h, \lambda}$ of a stochastic game with full state observation coincides with the value of the same game played in continuous time, with $\lambda$-discounted payoff $\int_0^{+\infty} \lambda e^{-\lambda t} g_t \, dt$; see \cite[\S 3]{Ney13}, \cite[\S 8]{SorVig16}, \cite[Remark~4]{Nov24a} for more details.
Hence, when it exists, the limit
$$\lim_{\lambda \to 0} \lim_{h \to 0} v_{h, \lambda}$$
can be interpreted as the asymptotic value of the game played in continuous time.

The papers \cite{Ney13} and \cite{SorVig16} discuss the existence of the evanescent value $\lim_{h \to 0} v_{h, \lambda}$ for a fixed discount factor $\lambda$. \cite{Ney13} proved its existence for finite games, while \cite{SorVig16} generalized this result to games in which the state and action spaces are general compact sets, and stage durations $h_n$ may depend on the stage number $n$. A corollary (see Proposition~\ref{cor1} below) of these results states that the asymptotic value $\lim_{\lambda \to 0} v_{1, \lambda}$ of a discrete-time stochastic game coincides with the asymptotic value $\lim_{\lambda \to 0} \lim_{h \to 0} v_{h, \lambda}$ of the same game played in continuous time. Games with stage duration have also been studied in \cite{CarRaiRosVie16}, \cite{Fab16}, \cite{Sor18}.

The main goal of this article is to study the \textbf{asymptotic value in a more general framework of zero-sum stochastic games with stage duration and public signals}. We begin by extending the definition of games with stage duration to games with deterministic public signals: the states are partitioned into a finite number of sets, and players are informed only of the set containing the current state. Notably, this article is the first to introduce and analyze games with stage duration in the context of public signals, whereas prior work has focused solely on fully observed or state-blind settings.

In \S\ref{twoexamples}, we present two simple examples showing that, if the state is not fully observed, then the limits $\lim_{\lambda \to 0} v_{h, \lambda}$ and $\lim_{\lambda \to 0} \lim_{h \to 0} v_{h, \lambda}$ are unrelated.

Our main result, \textbf{Theorem~\ref{thh2}} from \S\ref{jjjnt442}, gives an example of a \emph{finite} zero-sum stochastic game with deterministic public signals, in which the asymptotic value in continuous time ($\lim_{\lambda \to 0} \lim_{h \to 0} v_{h, \lambda}$) exists, while the asymptotic value in discrete time ($\lim_{\lambda \to 0} v_{1, \lambda}$) \emph{does not}. 
This demonstrates that moving from discrete to continuous time can have a regularizing effect, producing an asymptotic value even when the original discrete-time game lacks one. Importantly, these two limits always coincide for stochastic games with full state observation, highlighting the critical role of imperfect information in generating this phenomenon.

Let us briefly outline the main ideas used in the proof of Theorem~\ref{thh2}. The game we consider is similar to those in \cite{Zil16} and \cite{RenZil20}, for which it is known that the asymptotic value $\lim_{\lambda \to 0} v_{1, \lambda}$ does not exist. The existence of $\lim_{\lambda \to 0} \lim_{h \to 0} v_{h, \lambda}$ is proved by using \cite[Theorem~2]{Nov24a}: $\lim_{h \to 0} v_{h, \lambda}$ is a viscosity solution of a certain differential equation. We then use this result to compute explicitly $\lim_{\lambda \to 0} \lim_{h \to 0} v_{h, \lambda}$.

\section{Zero-sum stochastic games with deterministic public signals}

\label{repeatedGames}

A \emph{zero-sum stochastic game with deterministic public signals} is a  $7$-tuple $(A, \Omega, f, I, J, g, P)$, where: 
\begin{itemize}
	\item $A$ is the finite set of signals;
	\item $\Omega$ is the finite set of states;
	\item $f : \Omega \to A$ is the state-signaling function;
	\item $I$ is the finite set of actions of Player~1;
	\item $J$ is the finite set of actions of Player~2;
	\item $g : I \times J \times \Omega \to \R$ is the stage payoff function;
	\item $P : I \times J \times \Omega \to \Delta(\Omega)$ is the transition probability function.
\end{itemize}

For ease of notation, we write $P (\omega \mid i,j, \omega')$ instead of $P (i,j, \omega')(\omega)$.

The game $(A, \Omega, f, I, J, g, P)$ proceeds in stages as follows. Before the first stage, an initial state $\omega_1 \in \Omega$ is drawn according to a probability law $p_0$, and the players receive the signal $\alpha_1 = f(\omega_1)$. At each stage $n \in \N^*$:
\begin{enumerate}
	\item The current state is $\omega_n \in \Omega$. The players do not observe it, but they observe the signal $\alpha_n = f(\omega_n) \in A$ and recall each other’s pure actions at the previous stage.
	\item The players simultaneously choose their mixed actions: Player~1 chooses $x_n \in \Delta(I)$ and Player~2 chooses $y_n \in \Delta(J)$.
	\item A pure action $i_n \in I$ of Player~1 (respectively $j_n \in J$ of Player~2) is drawn according to $x_n$  (respectively $y_n$).
	\item Player~1 receives payoff $g_n = g(i_n, j_n, \omega_n)$, while Player~2 receives payoff $-g_n$. The next state $\omega_{n+1}$ is drawn according to the probability law $P(\cdot \mid i_n, j_n, \omega_n)$.
\end{enumerate}
The above description of the game is assumed to be common knowledge.

A state $\omega$ is called \emph{absorbing} if for any $i \in I, j \in J$ we have $P(\omega \mid i, j, \omega) = 1$.

Two special cases of the above construction are as follows:
\begin{enumerate}
\item If $A$ is a singleton, then $G$ is a 
\emph{state-blind stochastic game}; 
\item If
$\alpha_{n} = \omega_{n}$ for all $n \in \N^*$, then $G$ is a 
\emph{stochastic game with full state observation}, cf. \cite{Sha53}. \demo
\end{enumerate}

\emph{A history of length $t \in \N$} in the stochastic game $(A, \Omega, f, I, J, g, P)$ is
$(\alpha_1, i_1, j_1, \alpha_2, i_2, j_2, \ldots, \alpha_{t-1}, i_{t-1}, \allowbreak j_{t-1}, \alpha_t)$.
The set of all histories of length $t$ is 
$H_t := A \times (I \times J \times A)^{t-1}$.
A \emph{(behavior) strategy} of Player~1 (respectively Player~2) is a function 
$\sigma : \bigcup_{t \ge 1} H_t \to \Delta(I)$ 
(respectively $\tau : \bigcup_{t \ge 1} H_t \to \Delta(J)$). 
Players' strategies induce a probability distribution on the set $A \times (I \times J \times A)^{\N^*}$. (Indeed, strategies induce a probability distribution on the set $H_1$, then on the set $H_2$, etc. By the Kolmogorov extension theorem, this probability can be extended in a unique way to the set 
$A \times (I \times J \times A)^{\N^*}$). In particular, given an initial probability distribution $p \in \Delta(\Omega)$, strategies 
$\sigma : \bigcup_{t \ge 1} H_t \to \Delta(I), \tau : \bigcup_{t \ge 1} H_t \to \Delta(J)$, 
and the induced probability distribution $P^p_{\sigma, \tau}$ on $A \times (I \times J \times A)^{\N^*}$, we can consider the expectation $E^p_{\sigma, \tau}$ of a random variable on $\bigcup_{t \ge 1} H_t$.

We now specify how to compute the total payoff. Given a stochastic game $\Gamma = (A, \Omega, f, I, J, g, P)$ and a discount factor $\lambda \in (0,1]$, we consider the $\lambda$-discounted stochastic game $\Gamma_\lambda$, whose total payoff is $E^p_{\sigma, \tau} \left( \lambda \sum\nolimits_{n=1}^\infty (1-\lambda)^{n-1} g_n\right).$
The game $\Gamma_\lambda$ is said to have a value 
$v : \Delta(\Omega) \to \R$ if for all $p \in \Delta(\Omega)$ we have
$v(p) = \sup_{\sigma} \inf_{\tau} E^p_{\sigma, \tau} 
\left(\lambda \sum\nolimits_{n=1}^\infty (1-\lambda)^{n-1} g_n\right) = 
\inf_{\tau} \sup_{\sigma} E^p_{\sigma, \tau} 
\left(\lambda \sum\nolimits_{n=1}^\infty (1-\lambda)^{n-1} g_n\right).$
In our finite framework, the value always exists; see \cite[IV.3]{MerSorZam15}.



From $P$ we can compute the \emph{kernel}
$q : I \times J \times \Omega \to \R^{|\Omega|}$
defined by the expression
$$q(\omega \mid i,j, \omega') = 
\begin{cases}
P(\omega \mid i,j, \omega'), &\text{if } \omega \neq \omega';\\
P(\omega \mid i,j, \omega) - 1, &\text{if } \omega = \omega'.
\end{cases}$$
It will be convenient to define certain games directly in terms of their kernel. Note that for any fixed $\omega' \in \Omega$ we have
$\sum_{\omega \in \Omega} q(\omega \mid i,j, \omega') = 0$.

\begin{remark}[Difference with the standard definition of stochastic games with 
public signals]
\label{diff32}
Unlike the general case of games with random signals $P' : I \times J \times \Omega \to \Delta(\Omega \times A)$, the present framework considers the special case in which signals are deterministically given by a partition of the state set. We restrict attention to this setting because: (i) it is enough for our example; (ii) it is easier to define games with stage duration in this particular framework. \demo
\end{remark}

\section{Zero-sum stochastic games with stage duration and deterministic public signals}
\label{anti-anti}


\subsection{Definition of the model}
\label{def1}

In our particular setting of games with deterministic public signals, it is straightforward to define games with stage duration, since the signal depends only on the current state and is therefore independent of the stage duration.

\begin{definition}
\label{rty001}
Let $(A, \Omega, f, I, J, g, q)$ be a zero-sum stochastic game with deterministic public signals. The \emph{\mbox{$\lambda$-discounted} stochastic game with public signals and stage duration $h \in (0,1]$} is the $\lambda h$-discounted stochastic game $(A, \Omega, f, I, J, g, h q)$ with deterministic public signals. Denote by $v_{h,\lambda}$ the value of this game.
\end{definition}

Definition~\ref{rty001} provides an approximation of a continuous-time game. To see this, fix a small $h$ and consider a continuous-time game with payoff $g$ and generator $q$, in which the players are restricted to choosing actions only at the discrete times $0, h, 2h, \ldots$. Then, between times $n h$ and $(n+1) h$, the players receive the payoff $$\int_{n h}^{(n+1) h} \lambda e^{-\lambda t} g(t) \, dt \approx \lambda h (1 - \lambda h)^{n} g(n h),$$ and the state evolves according to the probability law $\exp(h q) \approx 1 + h q$. See \cite[\S 3]{Ney13}, \cite[\S 8]{SorVig16}, \cite[Remark~4]{Nov24a} for more details.

Thus the total payoff is $\sum\nolimits_{i = 1}^\infty \lambda h (1-\lambda h)^{i-1} g_i.$ 

We denote by $v_{0,\lambda}$ the limit $\lim_{h \to 0} v_{h,\lambda}$, whenever it exists. This limit is known to exist in the case of full state observation and in the state-blind case; see Remarks~\ref{xcv4_0}--\ref{xcv4_1}.

\begin{remark}[The case of full observation]
\label{xcv4_0}
Games with stage duration and full state observation were first introduced in \cite{Ney13} and subsequently studied in \cite{SorVig16}. According to \cite[Theorem~1]{Ney13}, the limit $v_{0,\lambda}$ exists and satisfies
$$v_{0,\lambda} = v_{1,\frac{\lambda}{1+\lambda}}.$$ A similar model of games with stage duration was considered in \cite{Sor18}.\demo
\end{remark}

\begin{remark}[State-blind case]
\label{xcv4_1}
State-blind games with stage duration were recently studied in \cite{Nov24a}. 
The result \cite[Theorem~2]{Nov24a} (see also \cite[Proposition~5.3]{Sor18}) shows that for any state-blind game
$(A, \Omega, f, I, J, g, q)$, the limit $v_{0,\lambda}(p)$ exists and is the unique viscosity solution of the partial differential equation
$$\lambda v(p) = \texttt{Val}_{I\times J} [\lambda g(i,j,p) + \langle p * q(i,j), \nabla v(p) \rangle],$$ where for each $i \in I, j \in J, \omega \in \Omega$ we have
$\left(p * q(i,j)\right)(\omega) := \sum\nolimits_{\omega'\in \Omega} p(\omega') \cdot q(\omega \mid i,j,\omega')$.
\demo
\end{remark}


Now, we recall some known results.

\begin{proposition}
\label{cor1}
(Follows from \cite[Theorem~1]{Ney13} and \cite{BewKoh76}). Let $(A, \Omega, f, I, \allowbreak J, g, q)$ be a finite stochastic game with full state observation. For a fixed discount factor $\lambda$, we have:
\begin{enumerate}
	\item For any $h \in (0,1]$, the limits $\lim\nolimits_{\lambda \to 0} v_{h, \lambda}$ and $\lim\nolimits_{\lambda \to 0} v_{1,\lambda}$ both exist and are equal: $$\lim_{\lambda \to 0} v_{h, \lambda} = \lim_{\lambda \to 0} v_{1,\lambda};$$
	\item 
  We have
$$\lim\limits_{h \to 0}
\left(\lim\limits_{\lambda \to 0} v_{h, \lambda}\right) = 
\lim\limits_{\lambda \to 0} v_{0, \lambda} = 
\lim\limits_{\lambda \to 0} v_{1,\lambda}.$$
\end{enumerate}
\end{proposition}

\begin{remark}[Analogue of Proposition~\ref{cor1} in the case of infinite $I, J,$ or $\Omega$]
\label{sbnke}
(Follows from \cite[Corollary~7.1]{SorVig16}).
Earlier in this paper, we assumed that the action sets $I, J$ and the set of states $\Omega$ are finite. 
In the infinite case, the existence of the asymptotic value (as $\lambda \to 0$) is not guaranteed; see \cite{Zil16}, \cite{RenZil20}, \cite{Vig13}. In this case, we have:
\begin{enumerate}
	\item Under some mild technical assumptions, $v_{0,\lambda} (\omega)$ exists and satisfies 
	$$v_{0,\lambda} (\omega) = v_{1,\frac{\lambda}{1+\lambda}} (\omega).$$
	\item If $\lim_{\lambda \to 0} v_{1,\lambda}$ exists, then so does $\lim_{\lambda \to 0} v_{0,\lambda}$, and we have
$$\lim\limits_{h \to 0}
\left(\lim\limits_{\lambda \to 0} v_{h, \lambda}\right) = 
\lim\limits_{\lambda \to 0} v_{0, \lambda} = 
\lim\limits_{\lambda \to 0} v_{1,\lambda}.$$
\item If $\lim\nolimits_{\lambda \to 0} v_{1,\lambda}$ does not exist, then neither does $\lim_{\lambda \to 0} v_{0,\lambda}$. \demo
\end{enumerate}
\end{remark}


\subsection{Examples where the asymptotic values in discrete and continuous time differ}
\label{twoexamples}

In this section, we briefly present two examples of games in which the asymptotic value (as the discount factor $\lambda$ approaches $0$) of a game played in discrete time does not coincide with the asymptotic value of the same game played in continuous time, even though both asymptotic values exist. As mentioned in Remark~\ref{sbnke}(2), this phenomenon cannot occur when the state is fully observed.

\begin{example}
\label{exxxx1}
The example is from Guillaume Vigeral (private communication). Consider the state-blind Markov decision process shown in Figure~\ref{fig1}.

\begin{figure}[h]
 \centering
	\begin{tikzpicture}[font = {\large}, node distance = 1.5cm,
node/.style = {circle, draw = black!100, very thick, minimum size = 13mm}]
	\node[node, label={[align=right]left:State $S_1$\\ 
	Initial probability $p_1$}] (S1) {$0$};
	\node[node, label={[align=left]right:State $S_2$\\ 
	Initial probability $p_2$}] (S2) [right=of S1]{$0$};
	\node[node, label={[align=right]left:State $S_3$}] (S3) [below=of S1]{$-1^*$};
	\node[node, label={[align=left]right:State $S_4$}] (S4) [below=of S2]{$+1^*$};

	\path[ultra thick, ->]  (S1) edge[bend left] 	node [above]	{$C$} (S2)
													(S2) edge[bend left] 	node [below]	{$C$} (S1)
													(S1) edge						 	node [right]	{$Q$} (S3)
													(S2) edge						 	node [left]	{$Q$} (S4);							
	\end{tikzpicture}
	\caption{The Markov decision process from Example~\ref{exxxx1}.}
	\label{fig1} 
\end{figure}

There are 4 states $S_i \; (i =1,2,3,4)$, with $S_3$ and $S_4$ being absorbing. The action space is $\{C, Q\}$. In states $S_1, S_2, S_3, S_4$, the payoffs are $0, 0, -1, +1$, respectively, independent of the action chosen by the single decision-maker (the maximizing player). Figure~\ref{fig1} shows the transitions between states. For example, if the current state is $S_1$ and Player~1 plays $C$, then the next state is $S_2$. Denote the initial probability distribution over the states by $p = (p_1,p_2)$, where $p_i$ is the probability that the initial state is $S_i$. 

A computation shows that
$$v_{1,\lambda}(p) = \max\{(p_2 - p_1)(1-\lambda), (p_1 - p_2)(1-\lambda)^2\} \xrightarrow{\lambda \to 0} |p_1 - p_2|.$$

A computation in Appendix~\ref{appenA} shows that
\[v_{0,\lambda}(p) = \frac{1}{1+\lambda} \max\{0, p_2-p_1\}
\xrightarrow{\lambda \to 0} \max\{0, p_2-p_1\}.
\demoo\]
\end{example}

\begin{remark}[Explanation of the expression for $v_{0,\lambda}$]
Appendix~\ref{appenA} shows that when $h < 1/2$, taking action $C$ moves the player's belief about being in state $S_1$ closer to $1/2$, but never beyond it. Therefore:
\begin{enumerate} 
	\item If $p_1 > 1/2$, it is never optimal for the player to choose $Q$, and thus he will always choose $C$. This gives the optimal payoff of $0$.
	\item If  $p_1 < 1/2$, the only optimal action is to always choose $Q$. This gives the optimal payoff of $\frac{\lambda}{1+\lambda} (p_2 - p_1)$.
	\item If $p_1 = 1/2$, any sequence of actions is optimal, and the optimal payoff is $0$.
\end{enumerate} 
See Figure~\ref{belief_update}. \demo
\end{remark}

\begin{figure}[h!]
\centering
\begin{tikzpicture}[x=10cm, y=1cm, >=Stealth, thick]

  \draw[line width=1pt] (0,0) -- (1,0);

  \foreach \x/\lab in {0/0, 0.5/{\tfrac{1}{2}}, 1/1} {
    \draw[line width=1.2pt] (\x,0.12) -- (\x,-0.12);
    \filldraw[black] (\x,0) circle (2.5pt);
    \node[below=8pt, font=\large] at (\x,0) {$p=\lab$};}

  \draw[->, line width=1pt] (0.15,0.4) -- (0.45,0.4);
  \draw[->, line width=1pt] (0.85,0.4) -- (0.55,0.4);

  \node[align=center, font=\large, text width=10cm] at (0.5,1.2)
    {\textbf{$p$ denotes the belief that the current state is $S_1$.}};

\end{tikzpicture}
\caption{Illustration of belief updating when $h < 1/2$: after choosing $C$, the belief $p$ drifts toward $1/2$ but never crosses it.}
\label{belief_update}
\end{figure}

\begin{example}
Strictly speaking, this example from \cite{Ren06} falls outside the scope of our model because it involves incomplete information: Player~1 observes the current state, whereas Player~2 does not. Nonetheless, we include it here for its illustrative value. The game has two states, $S_1$ and $S_2$. Each player has two actions, and the payoff matrices are
\begin{equation*}
\centering
	\begin{tabular}{|c|c|}
    \hline
    $1$ & $0$\\
    \hline
    $0$ & $0$\\
    \hline
	\end{tabular} 
  \text{  in state } S_1, \text{and } \;\;\;\;\quad
	\begin{tabular}{|c|c|}
    \hline
    $0$ & $0$\\
    \hline
    $0$ & $1$\\
    \hline
	\end{tabular} 
  \text{ in state } S_2.
\end{equation*}

For any pair of actions, the state remains unchanged with probability $1/2$. Denote by $v_\lambda(p)$ the $\lambda$-discounted value of this game, in which the initial state is $S_1$ with probability $p$ and $S_2$ with probability $1-p$. A computation shows that
$$v_\lambda(p) = \lambda \min\{p, 1-p\} + (1-\lambda) \frac{1}{2} \xrightarrow{\lambda \to 0} \frac{1}{2}.$$

\cite{CarRaiRosVie16} considers the same game in continuous time. More precisely, this article considers the continuous-time game with the infinitesimal generator
\begin{equation*}
  \centering
	\begin{tabular}{|c|c|}
    \hline
    $-\pi$ & $\pi$\\
    \hline
    $\pi$ & $-\pi$\\
    \hline
	\end{tabular},
\end{equation*}
where $\pi > 0$. 
The state variable evolves as it would in the continuous-time model, but players are allowed to act only at times $0, h, 2h,$ and so on.\footnote{See \cite[\S 2.1]{CarRaiRosVie16} for a more detailed description of the model.} Thus, for each fixed $h$, \cite{CarRaiRosVie16} studies a discrete-time game with incomplete information. As $h$ approaches $0$, the values of these games converge to the value $v_\lambda^{\texttt{cont}}(p)$ of the continuous-time game. A result in \cite[\S 3.3]{CarRaiRosVie16} states that
\[v_\lambda^{\texttt{cont}}(p) = \frac{1}{4} - \frac{(2p-1)^2}{4} \cdot \frac{\lambda}{\lambda + 4 \pi} \xrightarrow{\lambda \to 0} \frac{1}{4}.\demoo\]
\end{example}

The above two examples illustrate that, in general, the asymptotic values in discrete- and continuous-time settings might differ. In the rest of the paper, we construct an example in which another phenomenon appears: transitioning from discrete time ($h=1$) to continuous time ($h \to 0$) can restore the existence of an asymptotic value that is otherwise absent in discrete time.

\subsection[Main result: asymptotic value exists in continuous time but not in discrete time]
{Main result: asymptotic value exists in continuous time but not in discrete time}
\label{jjjnt442}

\begin{example}
\label{mwa256}
Consider the following two-player zero-sum stochastic game $G_1 = (A, \Omega, f, I, J, g, q)$. 
\begin{itemize}
	\item The state space is $\Omega = \{ \omega_1, \omega_2, \omega_3^*, \omega_4, \omega_5, \omega_6^*\}$. The states $\omega_3^*$ and $\omega_6^*$ are absorbing. 
	\item The stage payoff $g$ depends only on the state, taking the value $-1$ in states $\omega_1, \omega_2, \omega_3^*$ and $+1$ in states $\omega_4, \omega_5, \omega_6^*$.
	\item The signal set is $A = \{ \texttt{PLUS},\, \texttt{MINUS} \}$, and we have $$f(\omega_1) = f(\omega_2) = f(\omega_3^*) = \texttt{MINUS} \quad \text{ and } \quad f(\omega_4) = f(\omega_5) = f(\omega_6^*) = \texttt{PLUS}.$$
	\item In states $\omega_1, \omega_2, \omega_3^*$, Player~1’s action set is $\{T,B,Q\}$ and Player~2’s action set is $\{L,R\}$. In states $\omega_4, \omega_5, \omega_6^*$, Player~1’s action set is $\{T,M,B\}$ and Player~2’s action set is $\{L,M,R,Q\}$. The transition matrices for the non-absorbing states are shown in Tables~\ref{eenrn43}--\ref{eenrn44}.
\end{itemize}

\begin{table*}[h]
	\begin{minipage}{0.34\linewidth}
	\centering
	\begin{tabular}{|c|c|c|}
\hline
& $L$ & $R$\\
\hline
$T$ & $\omega_1$ & $\omega_2$\\
\hline
$B$ & $\omega_2$ & $\omega_1$ \\
\hline
$Q$ & $\omega_5$ & $\omega_5$\\
\hline
	\end{tabular}
	\caption{State $\omega_1$}
  \label{eenrn43}
	\end{minipage}%
	\begin{minipage}{0.75\linewidth}
	\centering
	\begin{tabular}{|c|c|c|}
\hline
& $L$ & $R$\\
\hline
$T$ & $\frac{1}{2} \omega_1 + \frac{1}{2} \omega_2$ & $\omega_2$\\
\hline
$B$ & $\omega_2$ & $\frac{1}{2} \omega_1 + \frac{1}{2} \omega_2$\\
\hline
$Q$ & $\omega_3^*$ & $\omega_3^*$\\
\hline
	\end{tabular} 
	\caption{State $\omega_2$}
  \label{eenrn41}
	\end{minipage}
\end{table*}
\begin{table*}[h]
	\begin{minipage}{0.34\linewidth}
	\centering
	\begin{tabular}{|c|c|c|c|c|}
\hline
& $L$ & $M$ & $R$ & $Q$\\
\hline
$T$ & $\omega_4$ & $\omega_5$ & $\omega_5$ & $\omega_2$\\
\hline
$M$ & $\omega_5$ & $\omega_4$ & $\omega_5$ & $\omega_2$\\
\hline
$B$ & $\omega_5$ & $\omega_5$ & $\omega_4$ & $\omega_2$\\
\hline
	\end{tabular}
	\caption{State $\omega_4$}
  \label{eenrn42}
	\end{minipage}%
	\begin{minipage}{0.75\linewidth}
	\centering
	\begin{tabular}{|c|c|c|c|c|}
\hline
& $L$ & $M$ & $R$ & $Q$\\
\hline
$T$ & $\frac{2}{3} \omega_4 + \frac{1}{3} \omega_5$ & $\omega_5$ & $\omega_5$ & 
$\omega_6^*$\\
\hline
$M$ & $\omega_5$ & $\frac{2}{3} \omega_4 + \frac{1}{3} \omega_5$ & $\omega_5$ & 
$\omega_6^*$\\
\hline
$B$ & $\omega_5$ & $\omega_5$ & $\frac{2}{3} \omega_4 + \frac{1}{3} \omega_5$ & 
$\omega_6^*$\\
\hline
	\end{tabular}
	\caption{State $\omega_5$}
  \label{eenrn44}
	\end{minipage}
\end{table*}

The initial probability distribution over the states is denoted by $p = (p_1,p_2,p_4,p_5),$ where $p_i$ is the probability that the initial state is $\omega_i$. See Figure~\ref{llmp44}.

\begin{figure}[H]
\centering
\hspace*{1.5cm} 
\begin{tikzpicture}
\large
\draw (3.5,0) circle [radius=0.35] node[label={[label distance=-0.1cm]left:$p_4$}] {$\omega_4$};
\draw (3.5,-1.2) circle [radius=0.35] node[label={[label distance=-0.1cm]left:$p_5$}] {$\omega_5$};
\draw (3.5,-2.4) circle [radius=0.35] node[] {$\omega_6^*$};
\draw (0,0) circle [radius=0.35] node[label={[label distance=-0.1cm]right:$p_1$}] {$\omega_1$};
\draw (0,-1.2) circle [radius=0.35] node[label={[label distance=-0.1cm]right:$p_2$}] {$\omega_2$};
\draw (0,-2.4) circle [radius=0.35] node[] {$\omega_3^*$};
\draw[color=red, line width=2pt, align=center] (3.5,-1.2) ellipse 
(1.5cm and 1.85cm) node[right=-60pt] {\hspace{0.2cm} Signal \texttt{PLUS}\\ 
\hspace{-0.1cm} Payoff $+1$ \\ 
\hspace{3cm} Player 1's actions: $T, M, B$\\
\hspace{3.4cm} Player 2's actions: $L, M, R, Q$};
\draw[color=blue, line width=2pt, align=center] (0,-1.2) ellipse 
(1.5cm and 1.85cm) node[left=40pt] {\hspace{0.75cm} Signal \texttt{MINUS}\\ 
\hspace{1.3cm} Payoff $-1$ \\
\hspace{-1.8cm} Player 1's actions: $T, B, Q$\\
\hspace{-1.3cm} Player 2's actions: $L, R$};
\end{tikzpicture}
\caption{Signaling structure on the states. The letter next to each state indicates its probability of being the initial state.}
\label{llmp44}
\end{figure}
 \demo
\end{example}
\begin{remark}
Note that the number of actions depends on the state. This is done to simplify the exposition. One could introduce redundant actions so that the set of actions is the same across all states. \demo
\end{remark}

We denote by $G^\lambda_1$ the game $G_1$ with $\lambda$-discounted payoff.

\begin{theorem}
\label{thh2}
For the game $G^\lambda_1$ from Example~\ref{mwa256}:
\begin{enumerate}
\item
$\lim\limits_{\lambda \to 0} v_{1,\lambda}$ does not exist.
\item For each $\lambda \in (0,1]$, the limit $v_{0,\lambda}$ exists.
\item $\lim\limits_{\lambda \to 0} v_{0,\lambda}$ exists.
\end{enumerate}
\end{theorem}


\section{The proof of Theorem~\ref{thh2}}
\label{thh20}

\subsection{Outline of the proof}

In \S\ref{prel3}, we present some preliminaries. In \S\ref{equiv2}, we recall that any stochastic game with public signals is equivalent to a stochastic game with full state observation, albeit with a larger state space. In \S\ref{SHA}, we recall the relationship between the Shapley equation and the game's value. 

\S\ref{thh21} and \S\ref{thh23} are devoted to the proof of Theorem~\ref{thh2}. 

In \S\ref{thh21}, we show that the pointwise limit $\lim_{\lambda \to 0} v_{1,\lambda}$ does not exist. This follows from the fact that $v_{1,\lambda}$ is also the value of a game studied in \cite[\S3.4]{RenZil20}.

In \S\ref{thh23}, we consider the game $G_{h}$ with stage duration $h$, whose $\lambda$-discounted value is $v_{h, \lambda}$. Our goal is to prove that the limits $v_{0,\lambda}$ and $\lim\nolimits_{\lambda \to 0} v_{0,\lambda}$ exist. The proof is fairly long, so we divide it into five parts:
\begin{itemize}
\item In \S\ref{rPART1}, for each $h \in (0, 1]$, we decompose $G_{h}$ into two state-blind ``half-games'', $G_{h}^-(k_-)$ and $G_{h}^+(k_+)$, whose $\lambda$-discounted values are denoted by $v_{h, \lambda}^- (k_-,p)$ and $v_{h, \lambda}^+ (k_+,p)$, respectively. Informally, $G_{h}^-(k_-)$ corresponds to the left part of $G_{h}$, with an absorbing state of payoff $k_-$ replacing the right half. Similarly, $G_{h}^+(k_+)$ corresponds to the right part of $G_{h}$, with an absorbing state of payoff $k_+$ replacing the left half. We prove Lemma~\ref{dwjbckdsjb} stating that $v_{h,\lambda}^-(k_-,p) = (k_- +1) v_{h,\lambda}^-(0,p) + k_-$ and $v_{h,\lambda}^+(k_+,p) = (1-k_+) v_{h,\lambda}^+(0,p) + k_+$. We denote $v_{h,\lambda}^+(p) := v_{h,\lambda}^+(0, p)$ and $v_{h,\lambda}^-(p) := v_{h,\lambda}^-(0, p)$.
\item In \S\ref{rPART2}, we aim to identify suitable candidates for $v_{0, \lambda}^- := \lim_{h \to 0} v^-_{h, \lambda}$ and $v_{0, \lambda}^+ := \lim_{h \to 0} v^+_{h, \lambda}$. To do this, we consider two partially observable Markov decision processes (POMDPs), $\widetilde G_h^- $ and $\widetilde G_h^+ $, each of which depends on $h\in (0,1]$ and has $\lambda$-discounted values $\widetilde v^{h,\lambda}_-$ and $\widetilde v^{h,\lambda}_+$, respectively. Proposition~\ref{xc44e} shows that these POMDPs have the same dynamics as the games $G_{h}^-$ and $G_{h}^+$, respectively. We then apply heuristic reasoning to identify candidates for $\lim_{h \to 0} \widetilde v^{h,\lambda}_-$ and $\lim_{h \to 0} \widetilde v^{h,\lambda}_+$, which we denote by $w_\lambda^-$ and $w_\lambda^+$, respectively.
\item In \S\ref{rPART3}, our goal is to prove Lemma~\ref{MainLemma}, namely, to show that $v_{0, \lambda}^- = w_\lambda^-$ and $v_{0, \lambda}^+ = w_\lambda^+$. A result in \cite{Nov24a} states that in order to verify it, it suffices to check that $w_\lambda^-$ and $w_\lambda^+$ are classical solutions of a partial differential equation from \cite{Nov24a}.
\item In \S\ref{rPART4}, we complete the proof. We explicitly compute $v_{0, \lambda}$ by combining the values $v_{0, \lambda}^-$ and $v_{0, \lambda}^+$, and then verify that $v_{0, \lambda}$ converges as $\lambda$ tends to $0$.
\item In \S\ref{rPART5}, we provide comments on the proof.
\end{itemize}

See Figure~\ref{jkqs5hygf54e}.

\begin{figure}[h]
  \centering 
  \begin{tikzpicture}[font = {\normalsize}, 
    node/.style = {rectangle, draw = black!100, very thick, minimum size =10mm},
    every text node part/.style={align=center}]
      \node[node] (S1) {Stochastic game $G_1$ with public signals,\\
                        $\lambda$-discounted game $G_1^\lambda$
                        and its value $v_{1,\lambda}$\\
                        (From Theorem~\ref{thh2})};
      \node[node] (S3) [below=1.2cm of S1]{Stochastic game $G_{h}$ with public signals,\\
                                           $\lambda$-discounted game $G_{h,\lambda}$
                                           and its value $v_{h,\lambda} $.};
      \node[node] (S4) [below=1.2cm of S3]{State-blind ``half-games'' 
                                           $G^-_{h} (k_-)$ and $G^+_{h} (k_+)$,\\
                                           $\lambda$-discounted games $G^-_{h,\lambda} (k_-)$
                                           and $G^+_{h,\lambda} (k_+)$, \\
                                           their values 
                                           $v^-_{h,\lambda} (k_-, p)$ and
                                           $v^+_{h,\lambda} (k_+, p)$.\\
                                           We also denote 
                                           $G^-_{h} := G^-_{h} (0), G^+_{h} := G^+_{h} (0),$ \\ $G^-_{h,\lambda} := G^-_{h,\lambda} (0), G^+_{h,\lambda} := G^+_{h,\lambda} (0),$ \\ $v^-_{h,\lambda} (p) := v^-_{h,\lambda} (0, p), v^+_{h,\lambda} (p) := v^+_{h,\lambda} (0, p),$ \\ $v^-_{0, \lambda} (p) := \lim_{h \to 0} v^-_{h, \lambda} (p), v^+_{0, \lambda} (p) := \lim_{h \to 0} v^+_{h, \lambda} (p)$. \\
                                           Introduced in \S\ref{rPART1} in order to prove that\\
                                           $\lim\limits_{\lambda \to 0}
                                           v_{0,\lambda}$ exists.};
      \node[node] (S5) [right=1cm of S4]{POMDPs
                                         $\widetilde G^-_{h}$ and $\widetilde G^+_{h}$,\\
                                         $\lambda$-discounted  POMDPs
                                         $\widetilde G^-_{h,\lambda}$
                                         and $\widetilde G^+_{h,\lambda}$, \\
                                         their values 
                                         $\widetilde v_-^{h,\lambda}$ and
                                         $\widetilde v_+^{h,\lambda}$.\\
                                         Used in \S\ref{rPART2}, because we can compute \\ heuristically
                                         $\widetilde v_-^{h,\lambda}$ and
                                         $\widetilde v_+^{h,\lambda}$,\\ and also because we have \\
                                         $v_{0,\lambda}^- = \lim\limits_{h \to 0} \widetilde v^{h,\lambda}_-$ and
                                         $v_{0,\lambda}^+ = \lim\limits_{h \to 0} \widetilde v^{h,\lambda}_+$.};

      \path[ultra thick, ->]  (S1) edge 		node [left]	  {stage duration} (S3)
                              (S3) edge 		node [left]	  {decomposition into ``half-games''} (S4)
                              (S4) edge     node [above]	{} (S5);
  \end{tikzpicture}
      \caption{This table outlines the plan of the proof, illustrating the connections between the different games in the proof and explaining why these games were introduced. It also serves as a reminder of the notation for the reader, in case of any confusion.}
  \label{jkqs5hygf54e} 
  \end{figure}

\subsection{Preliminaries}
\label{prel3}

\subsubsection{Any stochastic game with public signals is value-equivalent to a stochastic game with full state observation}
\label{equiv2}

We now recall a standard construction; see, e.g., \cite[\S 3]{MerSorZam15}. Given a zero-sum stochastic game $(A, \Omega, f, I, J, g, P)$ with deterministic public signals, we can define an analogous stochastic game with full state observation, whose state space is $\Delta(\Omega)$.

The payoff function  
$\widehat g : I \times J \times \Delta(\Omega) \to \R$ is defined by
$\widehat g (i,j,p) = \sum_{\omega \in \Omega} p(\omega) 
g(i, j, \omega)$.

Now let us define the transition probability function $\widehat P : I \times J \times \Delta(\Omega) \to \Delta_f(\Delta(\Omega))$. 
For each $\omega \in \Omega, p \in \Delta(\Omega), 
i \in I, j \in J$ denote 
$$P(\omega \mid i,j, p) = \sum_{\omega' \in \Omega} 
p(\omega') \cdot P(\omega \mid i,j, \omega').$$
If the players have belief $p \in \Delta(\Omega)$ 
about the current state, then after playing $(i, j) \in I \times J$,
the probability of receiving the signal $\alpha$ is
$$P( \alpha \mid i,j,p) = \sum_{\underset{f(\omega) = \alpha}{\omega\in\Omega}} P(\omega \mid i,j, p).$$

If the players have belief $p \in \Delta(\Omega)$ 
about the current state, then after playing $(i, j) \in I \times J$ and receiving the signal $\alpha \in A$, their posterior belief that the current state is $\omega$ is equal to 
$$P (\omega \mid i, j,\alpha, p) :=
\frac{P(\omega \mid i,j, p)}{P (\alpha \mid i,j,p)}$$
if $P (\alpha \mid i,j,p) \neq 0.$ This defines the posterior belief $P(i, j,\alpha, p)$ over the states.

Finally, the function
$\widehat P : I \times J \times \Delta(\Omega) \to \Delta_f(\Delta(\Omega))$
is defined by 
$$\widehat P (i,j, p) = \underset{P(\alpha \mid i,j,p) \neq 0}{\sum_{\alpha \in A}} P(\alpha \mid i,j,p) \cdot P(i,j,\alpha,p).$$

We can now consider the game $(\widehat A, \Delta(\Omega), \widehat f, I, J, \widehat g, \widehat P)$ with full state observation. Namely, $\widehat A = \Delta(\Omega)$ and $\widehat f(p) = p$ for any $p \in \Delta(\Omega)$. The following proposition follows directly from the above construction.

\begin{proposition} 
For any initial state distribution, the value of the game $(A, \Omega, f, I, J, g, P)$ with deterministic public signals coincides with the value of the game  $(\widehat A, \Delta(\Omega), \widehat f, I, J, \widehat g, \widehat P)$ with full state observation.
\end{proposition}


\subsubsection{Shapley operator}

\label{SHA}

\begin{proposition}
  \label{shapley_dop}
  Consider a stochastic game $(A, \Omega, f, I, J, g, P)$ with public signals, where $P$ is a probability distribution function. The value $v_\lambda(p)$ of this game exists and is the unique fixed point of the operator
  \begin{align}
    \label{aaa33r}
    &T: \{\text{Continuous functions on } \Delta(\Omega)\} \to 
    \{\text{Continuous functions on } \Delta(\Omega)\}, \notag \\
    &v_\lambda(p) \mapsto \texttt{Val}_{I \times J} 
    \bigg[\lambda \widehat g (i,j,p) + (1-\lambda) 
    \sum_{p' \in \Delta(\Omega)} \widehat P (p' \mid i,j,p) \cdot v_\lambda(p')\bigg].
  \end{align}
\end{proposition}

This proposition follows directly from \cite[Theorem~IV.3.2]{MerSorZam15}.

\begin{remark}
\label{onze}
  Note that the sum over $\Delta(\Omega)$ in \eqref{aaa33r} is finite, since the number of signals is finite. \demo
\end{remark}

\subsection{Non-existence of $\lim_{\lambda \to 0} v_{1,\lambda}$}
\label{thh21}

Note that, as long as the initial state is deterministic, $G_1$ is equivalent to a state-blind game, in which the signal on the state is not observed. Indeed, at each point in time, either the current state is absorbing ($\omega_3$ or $\omega_6$) or the signal can be determined by examining the parity of the number of times $Q$ has been played, making it irrelevant whether the signal is observed. This argument does not apply if the stage duration is less than $1$, because then there is a positive probability that the state remains the same when $Q$ is played.

We prove that $G^\lambda_1$ has the same dynamics as the state-blind game from \cite[\S3.4]{RenZil20}, which does not have an asymptotic value. This implies that $\lim_{\lambda \to 0} v_{1,\lambda}$ does not exist. More precisely, we consider the game from \cite[\S3.4]{RenZil20}, with $\alpha = 1/2, \beta = 1/3$, and we rename the states $(1,0), (1,1), (2,0), (2,1)$ as $\omega_1, \omega_2, \omega_4, \omega_5$, respectively.

To verify that $G^\lambda_1$ has the same dynamics as the game from \cite[\S3.4]{RenZil20}, we consider two cases.

\noindent \textbf{Case 1:} the initial probability is $p = (p_1,p_2,0,0)$. By Proposition~\ref{shapley_dop}, the value $v_{1,\lambda}(p)$ of $G_1^\lambda(p)$ and the value of the $\lambda$-discounted game from \cite[\S3.4]{RenZil20} both satisfy the equation
$$V_{\lambda} (p) = 
- \lambda + (1 - \lambda) \max\left\{
\frac{1}{2} V_\lambda(p') + 
\frac{1}{2} V_\lambda(p'') \; ; \;
p_1 V_\lambda(p''') + p_2 (-1)
\right\},$$
where 
$$p' = \left(p_1 + \frac{p_2}{2}, \frac{p_2}{2},0,0\right); \quad
p'' = \left(0,1,0,0\right); \quad
p''' = \left(0,0,0,1\right).$$
For the game $G_1$, the equation takes this form because Player 2's optimal strategy in $G_1$ is to play the mixed action $\frac{1}{2} L + \frac{1}{2} R$. For the game from \cite[\S3.4]{RenZil20}, the equation is straightforward to compute because only Player~1 can influence the game in states $\omega_1, \omega_2$.

\noindent \textbf{Case 2:} the initial probability is $p = (0,0,p_4,p_5)$. Similarly, the value $v_{1,\lambda}(p)$ and the value of the $\lambda$-discounted game from \cite[\S3.4]{RenZil20} both satisfy the equation
$$V_{\lambda} (p) = 
\lambda + (1 - \lambda) \min\left\{
\frac{1}{3} V_\lambda(p') + \frac{2}{3} V_\lambda(p'') \; ; \;
p_4 V_\lambda(p''') + p_5 (+1)
\right\},$$
where 
$$p' = \left(0,0,p_4+\frac{2 p_5}{3},\frac{p_5}{3}\right); \quad
p'' = \left(0,0,0,1\right); \quad
p''' = \left(0,1,0,0\right).$$

\subsection{Existence of $\lim\nolimits_{\lambda \to 0} v_{0,\lambda}$}
\label{thh23}

\subsubsection{Decomposition of the game into two ``half-games''}
\label{rPART1}

\begin{definition}
Fix $k_-, k_+ \in [-1,+1]$. The \emph{state-blind game $G_1^-(k_-)$} (respectively \emph{$G_1^+(k_+)$}) has four states: $\omega_1, \omega_2, \omega_3^*, k_-^*$ (respectively $\omega_4, \omega_5, \omega_6^*, k_+^*$). The stage payoff depends only on the state and is equal to $-1$ in states $\omega_1, \omega_2, \omega_3^*$, $+1$ in states $\omega_4, \omega_5, \omega_6^*$, $k_-$ in state $k_-^*$, and $k_+$ in state $k_+^*$. The states $\omega_3^*, \omega_6^*, k_-^*, k_+^*$ are absorbing. 

\noindent In $G_1^-(k_-)$, the actions of Player~1 are $T, B, Q$, and the actions of Player~2 are $L$, $R$. The transition matrices for non-absorbing states are given by Tables~\ref{eenrn43}--\ref{eenrn41}, but in Table~\ref{eenrn43} one needs to replace $\omega_5$ with $k_-^*$. 

\noindent In $G_1^+(k_+)$, the actions of Player~1 are $T, M, B$, and the actions of Player~2 are $L, M, R, Q$. The transition matrices for non-absorbing states are given by Tables~\ref{eenrn42}--\ref{eenrn44}, but in Table~\ref{eenrn42} one needs to replace $\omega_2$ with $k_+^*$.
\end{definition}

Informally, $G_{1}$ can be seen as a combination of two "half-games" $G_1^-(k_-)$ and $G_{1}^+(k_+)$; see Figure~\ref{xf34gl66h}.

\begin{figure}[h]
\centering 
\begin{tikzpicture}[framed]
\large
\draw (3.5,0) circle [radius=0.35] node {$\omega_4$};
\draw (3.5,-1.2) circle [radius=0.35] node {$\omega_5$};
\draw (3.5,-2.4) circle [radius=0.35] node {$\omega_6^*$};
\draw (0,0) circle [radius=0.35] node {$\omega_1$};
\draw (0,-1.2) circle [radius=0.35] node {$\omega_2$};
\draw (0,-2.4) circle [radius=0.35] node {$\omega_3^*$};
\draw[color=red, line width=2pt, align=center] (3.5,-1.2) 
ellipse (1.5cm and 1.85cm) node[right=50pt] 
{Signal \texttt{PLUS}\\ Payoff $+1$};
\draw[color=blue, line width=2pt, align=center] (0,-1.2) 
ellipse (1.5cm and 1.85cm) node[left=50pt] 
{Signal \texttt{MINUS}\\ Payoff $-1$};
\end{tikzpicture}\\
\begin{tikzpicture}
\node at (2.3,-6.3) {\normalsize Game $G_{1}$ with two public signals};
\coordinate (a1) at (-1.6,-6);
\coordinate (a2) at (-1.6,-8);
\coordinate (b1) at (6.6,-6);
\coordinate (b2) at (6.6,-8);
\draw[->,very thick] (a1) -- (a2);
\draw[->,very thick] (b1) -- (b2);
\end{tikzpicture}\\
	\subcaptionbox{State-blind ``half-game'' $G^-_{1}(k_-)$.}
	[.41\linewidth]{
\begin{tikzpicture}[framed]
\large
\draw (1.2,-1.2) circle [radius=0.35] node {$k_-^*$};
\draw (0,0) circle [radius=0.35] node[label={left:Prob. $1-p_-$}] {$\omega_1$};
\draw (0,-1.2) circle [radius=0.35] node[label={left:Prob. $p_-$}] {$\omega_2$};
\draw (0,-2.4) circle [radius=0.35] node{$\omega_3^*$};
\end{tikzpicture}
}
\hspace{0.07\linewidth}
  \subcaptionbox{State-blind ``half-game'' $G^+_{1}(k_+)$.}
[0.41\linewidth]{
\begin{tikzpicture}[framed]
\large
\draw (1.2,0) circle [radius=0.35] node[label={right:Prob. $1-p_+$}] {$\omega_4$};
\draw (1.2,-1.2) circle [radius=0.35] node[label={right:Prob. $p_+$}] {$\omega_5$};
\draw (1.2,-2.4) circle [radius=0.35] node {$\omega_6^*$};
\draw (0,-1.2) circle [radius=0.35] node {$k_+^*$};
\end{tikzpicture}
  }
\caption{From the two-signal game $G_{T^\infty}$, one obtains the state-blind ``half-games'' $G^-_1(k_-)$ and $G^+_1(k_+)$.}
\label{xf34gl66h}
\end{figure}

Now, we introduce some notation. As before, $h \in (0,1]$. 

\begin{enumerate}
\item $G_{h}$ (respectively $G_{h}^-(k_-)$, $G_{h}^+(k_+)$) is the game $G_1$ (respectively $G_{1}^-(k_-)$, $G_{1}^+(k_+)$) with stage duration~$h$.
\item $G_{h, \lambda}$ (respectively $G_{h, \lambda}^-(k_-)$, $G_{h, \lambda}^+(k_+)$) is the game $G_{h}$ (respectively $G_{h}^-(k_-)$, $G_{h}^+(k_+)$) with $\lambda$-discounted payoff.
\item $v_{h, \lambda}(p)$ is the value of $G_{h, \lambda}$ with the initial distribution $p = (p_1, p_2, p_4, p_5)$, where $p_i$ is the probability that the initial state is $\omega_i$.
\item $v_{h}^-(k,p)$ (respectively, $v_{h}^+(k,p)$) is the value of $G_{h,\lambda}^-(k)$ (respectively, $G_{h,\lambda}^+(k)$) with $\lambda$-discounted payoff, when the initial state is $\omega_2$ (respectively, $\omega_5$) with probability $p \in [0,1]$ and $\omega_1$ (respectively, $\omega_4$) with probability $1-p$.
\end{enumerate}

\vspace{0.23cm}

We denote $p_- = (0,1,0,0)$ and $p_+ = (0,0,0,1)$. It is straightforward to obtain that if $p = (p_1, p_2, 0, 0)$, then $v_{h, \lambda}(p) = v_{h, \lambda}^- \left(v_{h, \lambda} (p_+), p_2 \right)$. Similarly, if $p = (0, 0, p_3, p_4)$, then $v_{h, \lambda}(p) = v_{h, \lambda}^+ \left(v_{h, \lambda} (p_-), p_4 \right)$. Therefore, we focus on studying $v_{h}^-(k,p)$ and $v_{h}^+(k,p)$. First, we prove a simple lemma.

\begin{lemma}
\label{dwjbckdsjb}
  For any $p \in [0,1]$ and any $k_-, k_+ \in [-1,1]$,
  we have 
  $$v_{h,\lambda}^-(k_-,p) = (k_- +1) v_{h,\lambda}^-(0,p) + k_- \quad \text{ and } \quad v_{h,\lambda}^+(k_+,p) = (1-k_+) v_{h,\lambda}^+(0,p) + k_+.$$
\end{lemma}

\begin{proof}
  Note that as long as $k_- \ge -1$, the objective of the player is to reach the state $k_-$ as quickly as possible. So, for any $k_- \ge -1$, the game $G_{h,\lambda}^-(k_-)$ is an affine transformation of $G_{h,\lambda}^-(0)$, and thus any optimal strategy in $G_{h,\lambda}^-(0)$ remains optimal in $G_{h,\lambda}^-(k_-)$. Thus there exists $\alpha \in [0, 1]$ such that for any $k_- \in [-1,1]$ we have
  $v_{h,\lambda}^-(k_-,p) = \alpha k_- + (-1) (1-\alpha)$. 
  By taking $k_- = 0$, we obtain
  $$v_{h,\lambda}^-(0,p) = -1 + \alpha
  \iff \alpha = v_{h,\lambda}^-(0,p) + 1.$$

  Analogously, any optimal strategy in $G_{h,\lambda}^+(0)$ remains optimal in $G_{h,\lambda}^+(k_+)$, as long as $k_+ \le 1$. Thus there exists $\beta \in [0, 1]$ such that for any $k_+ \in [-1,1]$ we have $v_{h,\lambda}^+(k_+,p) = \beta k_+ + 1 (1-\beta)$. By taking $k_+ = 0$, we obtain
  $$v_{h,\lambda}^+(0,p) = 1 - \beta
  \iff \beta = 1 - v_{h,\lambda}^+(0,p).$$
  From this, the assertion of the lemma follows immediately.
\end{proof}

\begin{remark}
This proof is based on ideas developed in the article \cite{SorVig15}.
\end{remark}

Hence, if we want to find the limits
$$v_{0,\lambda}^-(k_-,p) := \lim\limits_{h \to 0} 
v_{h,\lambda}^-(k_-,p) \quad \text{ and } \quad 
v_{0,\lambda}^+(k_+,p) :=
\lim\limits_{h \to 0} v_{h,\lambda}^+(k_+,p),
$$
where $k_-, k_+ \in [-1,1]$, it is sufficient to consider games
$G_{h}^- := G_{h}^-(0)$ and
$G_{h}^+ := G_{h}^+(0)$,
$\lambda$-discounted games 
$G_{h,\lambda}^- := G_{h,\lambda}^-(0)$ and
$G_{h,\lambda}^+ := G_{h,\lambda}^+(0)$, 
and their values 
$v_{h,\lambda}^-(p) := v_{h,\lambda}^-(0, p)$ and
$v_{h,\lambda}^+(p) := v_{h,\lambda}^+(0, p)$, respectively.
Thus our goal is to find
$$v_{0, \lambda}^-(p) := \lim\limits_{h \to 0} 
v_{h, \lambda}^-(p) \quad \text{ and } \quad
v_{0, \lambda}^+(p) :=
\lim\limits_{h \to 0} v_{h, \lambda}^+(p).$$
In the next lemma, we will determine these two functions.

\begin{lemma}
  We have
  \begin{align*}
  &1) \;
v_{0,\lambda}^- (p) = w_\lambda^- (p) := 
\begin{cases}
 -\frac{p + \lambda}{1+\lambda},
& \text{if } p < \frac{4\lambda + 2}{4\lambda + 3}; \\
-1 + 
\frac{ (4 \lambda)^{4\lambda/3}}
{(1+\lambda) (3 + 4\lambda)^{1+(4\lambda/3)}} 
\left(3 p - 2\right)^{-4\lambda/3},&
\text{if } p \ge \frac{4\lambda + 2}{4\lambda + 3}.
\end{cases}\\
  &2) \;
v_{0,\lambda}^+ (p) = w_\lambda^+ (p) := 
\begin{cases}
\frac{p + \lambda}{1+\lambda},
& \text{if } p < \frac{9\lambda + 6}{9\lambda + 8}; \\
1 - \frac{2 (18 \lambda)^{9\lambda/8}}
{(1+\lambda) (8 + 9 \lambda)^{1+(9\lambda/8)}} 
\left(8 p - 6\right)^{-9\lambda/8},&
\text{if } p \ge \frac{9\lambda + 6}{9\lambda + 8}.
\end{cases}
  \end{align*}
\label{MainLemma}
\end{lemma}

This lemma is very important in the proof of Theorem~\ref{thh2}. First (in \S\ref{rPART2}), we explain why $w_\lambda^+$ and $w_\lambda^-$ are good candidates for $v_{0,\lambda}^+$ and $v_{0,\lambda}^-$ respectively. Afterwards (in \S\ref{rPART3}), we verify that we indeed have $v_{0,\lambda}^+ = w_\lambda^+$ and $v_{0,\lambda}^- = w_\lambda^-$.

\subsubsection{Finding good candidates for $v_{0,\lambda}^+$ and $v_{0,\lambda}^-$}
\label{rPART2}

In this section, we provide a heuristic proof of Lemma~\ref{MainLemma}. The formal proof appears in \S\ref{rPART3}.

\begin{proof}[Heuristic justification of Lemma~\ref{MainLemma}]
\ \\
We consider two families (each parametrized by $h\in (0,1]$) of two-signal POMDPs, $\widetilde G_h^-$ and $\widetilde G_h^+$. These POMDPs are constructed so that they have the same dynamics as the games $G_{h,\lambda}^-$ and $G_{h,\lambda}^+$, respectively; see Proposition~\ref{xc44e} below.

\begin{definition}
For a fixed $h\in (0,1]$, the \emph{POMDP $\widetilde G_h^-$} (respectively \emph{$\widetilde G_h^+$}) has action space $\{C, Q\}$, state space $\{\omega_1, \omega_2, \omega_3^*, 0^*\}$ (respectively $\{\omega_4, \omega_5, \omega_6^*, 0^*\}$), and signal space $\{\alpha,\beta\}$. The payoff is independent of the actions: the states $\omega_1, \omega_2, \omega_3^*$ have payoff $-h$, the states $\omega_4, \omega_5, \omega_6^*$ have payoff $+h$, and the state $0^*$ has payoff $0$. The states $\omega_3^*, \omega_6^*, 0^*$ are absorbing. The transition probabilities are described by Figures~\ref{transitions_half} and \ref{transitions_halff2}. The arrow from state $s_1$ to state $s_2$ with label $(X, p, \gamma)$ indicates that if the current state is $s_1$ and the player chooses action $X$, then with probability $p$ the state moves to $s_2$ and the player receives the signal $\gamma$. 

\noindent In $\widetilde G_h^-$, the goal of the player is to \textbf{maximize} the $\lambda$-discounted payoff.

\noindent In $\widetilde G_h^+$, the goal of the player is to \textbf{minimize} the $\lambda$-discounted payoff.
\end{definition}

\begin{figure}[h]
\centering 
	\subcaptionbox{\label{transitions_half} Transitions in $\widetilde G_h^-$.}
	[.45\linewidth]{
		  \begin{tikzpicture}[font = {\normalsize}, 
    node/.style = {font = \large, circle, draw = black!100, very thick, minimum size =10mm},
    every text node part/.style={align=center}]
      \node[node] (M) {$\pmb{\omega_1}$};
      \node[node] (MM) [below=2.3cm of M]{$\pmb{\omega_2}$};
      \node[node] (PP) [right=2.1cm of M]{$\pmb{0^*}$};
      \node[node] (MS) [below=1.5cm of MM]{$\pmb{\omega_3^*}$}; 
      \node[draw=none,fill=none] (N1) [below left=0cm and 0cm of MM]
{\textbf{\underline{belief \;$\pmb{p}$}}};
      \node[draw=none,fill=none] (N2) [below left=0cm and 0cm of M]
{\textbf{\underline{belief \;$\pmb{1-p}$}}};

      \path[ultra thick, ->]  (M) edge[bend left]     node [right]	
                                {$\pmb{(C,h/2,\beta)}$} (MM);
      \path[ultra thick, ->]  (M) edge[loop above]     node [above]	
                                {$\pmb{(C,1/2,\alpha)}$} ();
      \path[ultra thick, ->]  (M) edge[loop left]     node [left]	
                                {$\pmb{(C,(1-h)/2,\beta)}$} ();
      \path[ultra thick, ->]  (MM) edge[bend left]     node [left]	
                                {$\pmb{(C,h/4,\alpha)}$} (M);
      \path[ultra thick, ->]  (MM) edge     node [left]	
                                {$\pmb{(Q,h,\alpha)}$} (MS);
      \path[ultra thick, ->]  (MM) edge[out=320,in=290,looseness=7]     node [right]	
                                {$\pmb{(Q,1-h,\beta)}$} (MM);
      \path[ultra thick, ->]  (MM) edge[loop left]     node [left]	
                                {$\pmb{(C,(2-h)/4,\alpha)}$} ();
      \path[ultra thick, ->]  (MM) edge[loop right]     node [right]	
                                {$\pmb{(C,1/2,\beta)}$} ();
      \path[ultra thick, ->]  (M) edge     node [below]	
                                {$\pmb{(Q,h,\alpha)}$} (PP);
      \path[ultra thick, ->]  (M) edge[out=10,in=40,looseness=7]     node [above right = 0cm and -0.7cm]	
                                {$\pmb{(Q,1-h,\beta)}$} (M);
  \end{tikzpicture}
}%
\hspace{0.07\linewidth}
\subcaptionbox{\label{transitions_halff2} Transitions in $\widetilde G_h^+$.}
[0.45\linewidth]{
	  \begin{tikzpicture}[font = {\normalsize}, 
    node/.style = {font = \large, circle, draw = black!100, very thick, minimum size =10mm},
    every text node part/.style={align=center}]
      \node[node] (P)  {$\pmb{\omega_4}$};
      \node[node] (MM) [left=2.1cm of P]{$\pmb{0^*}$};
      \node[node] (PP) [below=2.3cm of P]{$\pmb{\omega_5}$};
      \node[node] (PS) [below=1.5cm of PP]{$\pmb{\omega_6^*}$};
      \node[draw=none,fill=none] (N1) [below right=0cm and 0cm of PP]
{\textbf{\underline{belief \;$\pmb{p}$}}};
      \node[draw=none,fill=none] (N1) [below right=0cm and 0cm of P]
{\textbf{\underline{belief \;$\pmb{1-p}$}}};

      \path[ultra thick, ->]  (PP) edge[bend right]     node [right]
                                {$\pmb{(C,2h/9,\alpha)}$} (P);
      \path[ultra thick, ->]  (P) edge[loop above]     node [above]	
                                {$\pmb{(C,1/3,\alpha)}$} ();
      \path[ultra thick, ->]  (P) edge[bend right]     node [left]	
                                {$\pmb{(C,2h/3,\beta)}$} (PP);
      \path[ultra thick, ->]  (P) edge[loop right]     node [right]	
                                {$\pmb{(C,2(1-h)/3,\beta)}$} ();
      \path[ultra thick, ->]  (PP) edge     node [right]	
                                {$\pmb{(Q,h,\alpha)}$} (PS);
      \path[ultra thick, ->]  (PP) edge[out=220,in=250,looseness=7]     node [left]	
                                {$\pmb{(Q,1-h,\beta)}$} (PP);
      \path[ultra thick, ->]  (PP) edge[loop left]     node [left]
                                {$\pmb{(C,2/3,\beta)}$} ();
      \path[ultra thick, ->]  (PP) edge[loop right]     node [right]	
                                {$\pmb{(C, (3-2h)/9,\alpha)}$} ();
      \path[ultra thick, ->]  (P) edge     node [below]	
                                {$\pmb{(Q,h,\alpha)}$} (MM);
      \path[ultra thick, ->]  (P) edge[out=170,in=140,looseness=7]     node [above left = 0cm and -0.7cm]	
                                {$\pmb{(Q,1-h,\beta)}$} (P);
  \end{tikzpicture}
}%
\caption{Transitions.}
\end{figure}

\begin{remark}
\label{irenNeman2}
Assume that in $\widetilde G_h^-$ (respectively $\widetilde G_h^+$), the current state is $\omega_2$ (respectively $\omega_5$) with probability $p$ and $\omega_1$ (respectively $\omega_4$) with probability $1-p$.
In the optimal strategy, the player first plays $C$ to reduce his belief $\widetilde p$ that the current state is $\omega_2$ (respectively, $\omega_5$) to a sufficiently small value, and then starts playing $Q$. Note that the player’s belief that the current state is $\omega_2$ (respectively, $\omega_5$) after playing $C$ can be computed: it is $p - \frac{h p}{2}$ (respectively $p - \frac{2 h p}{3}$) with probability $1/2$ (respectively $1/3$), and $p + h - h p$ with probability $1/2$ (respectively $2/3$). See Figures~\ref{xn1xn1} and \ref{xn1xn2}. \demo
\end{remark}

\begin{figure}[h]
\centering 
\subcaptionbox{\label{xn1xn1} The transition starting from belief~$p$ when the maximizing player chooses $C$ in~$\widetilde G_h^-$.}
[0.45\linewidth]{
\begin{tikzpicture}
\fill (0,0) circle (3pt) node[below] 
{$\phantom{h}\mathbf{p}\phantom{h}$} (2,0) circle (3pt) 
node[below] 
{$\mathbf{p + h - hp}$} (-3,0) circle (3pt) node[below] 
{$\mathbf{p - \frac{hp}{2}}$};
\draw[-Stealth, line width = 2pt] (0,0) to [out=120,in=70] (-3,0) 
node
[label={[align=center, label distance=0.7cm]70:
$\mathbf{Prob. \frac{1}{2}}$}] {};
\draw[-Stealth, line width = 2pt] (0,0) to [out=70,in=120] (2,0) node
[label={[align=center, label distance=0.35cm]120:
$\mathbf{Prob. \frac{1}{2}}$}] {};
\draw[line width = 2pt] (-4,0)--(4,0);
\end{tikzpicture}}
\hspace{0.07\linewidth}
	\subcaptionbox{\label{xn1xn2} The transition starting from belief~$p$ when the minimizing player chooses $C$ in~$\widetilde G_h^+$.}
	[.45\linewidth]{
\begin{tikzpicture}
\fill (0,0) circle (3pt) node[below] 
{$\phantom{h}\mathbf{p}\phantom{h}$} (2,0) circle (3pt) 
node[below] 
{$\mathbf{p + h - hp}$} (-3,0) circle (3pt) node[below] 
{$\mathbf{p - \frac{2 hp}{3}}$};
\draw[-Stealth, line width = 2pt] (0,0) to [out=120,in=70] (-3,0) 
node
[label={[align=center, label distance=0.7cm]70:
$\mathbf{Prob. \frac{1}{3}}$}] {};
\draw[-Stealth, line width = 2pt] (0,0) to [out=70,in=120] (2,0) node
[label={[align=center, label distance=0.35cm]120:
$\mathbf{Prob. \frac{2}{3}}$}] {};
\draw[line width = 2pt] (-4,0)--(4,0);
\end{tikzpicture}}
\caption{Dynamics of the belief.}
\end{figure}

We denote by $\widetilde G_{h,\lambda}^-$ (respectively $\widetilde G_{h,\lambda}^+$) the POMDP $\widetilde G_{h}^-$ (respectively $\widetilde G_{h}^+$) with $\lambda$-discounted payoff, and we denote by $\widetilde v^{h,\lambda}_-$ (respectively $\widetilde v^{h,\lambda}_+$) the value of $\widetilde G_{h,\lambda}^-$ (respectively $\widetilde G_{h,\lambda}^+$).

\begin{proposition}
\label{xc44e}
For each fixed $\lambda \in (0,1]$ and $h \in [0,1]$, we have: 
\begin{enumerate}
\item The value $\widetilde v^{h, \lambda}_-$ of $\widetilde G_{h,\lambda}^-$ coincides with the value $v_{h, \lambda}^-$ of $G_{h,\lambda}^-$.
\item The value $\widetilde v^{h, \lambda}_+ $ of $\widetilde G_{h,\lambda}^+$ coincides with the value $v_{h, \lambda}^+ $ of $G_{h,\lambda}^+$.
\end{enumerate}
\end{proposition}

This proposition follows from Proposition~\ref{shapley_dop}: both $\widetilde v^{h, \lambda}_- $ and $v_{h, \lambda}^- $ are the unique solutions of a functional equation. The same is true for $\widetilde v^{h, \lambda}_+$ and $v_{h, \lambda}^+ $. We do not give a more detailed proof, since this entire subsection is dedicated to a heuristic search for a solution and thus we do not need to give precise proofs.

Note that in $\widetilde G_h^-$, if the player takes action $C$, then he receives the signal $\alpha$ with probability $1/2$ and the signal $\beta$ with probability $1/2$, while in $\widetilde G_h^+$, if the player takes action~$C$, then he receives the signal $\alpha$ with probability $1/3$ and the signal $\beta$ with probability $2/3$.

\noindent \textbf{The value $\pmb{\widetilde v^{h,\lambda}_-}$ of the 
$\pmb{\lambda}$-discounted POMDP 
$\pmb{\widetilde G_{h,\lambda}^-}$ (heuristic computation).} 
We denote by $p$ (respectively $1-p$) the probability that the initial state is $\omega_2$ (respectively $\omega_1$).

By Proposition~\ref{shapley_dop} we have
\begin{equation}
\widetilde v^{h,\lambda}_-(p) = 
-\lambda h + (1-\lambda h)\max\bigg\{
\underbrace{- h p }_
{\substack{\text{Player chooses } Q, \\ 
\text{the signal is } \alpha}}
+
\underbrace{(1-h) \widetilde v^{h,\lambda}_-(p)}_
{\substack{\text{Player chooses } Q, \\ 
\text{the signal is } \beta}}
;
\underbrace{\frac{1}{2} \widetilde v^{h,\lambda}_-
\left(p- \frac{h p}{2}\right)}_
{\substack{\text{Player chooses } C, \\ 
\text{the signal is } \alpha}}
+
\underbrace{\frac{1}{2} 
\widetilde v^{h,\lambda}_-(p + h - h p)}_
{\substack{\text{Player chooses } C, \\ 
\text{the signal is } \beta}} \bigg\}.
\label{eqE1}
\end{equation}

From the structure of the POMDP, it is natural to assume that there exists $p^* \in [0,1]$ such that the player prefers action $Q$ when $p \le p^*$, and action $C$ when $p > p^*$.
In that case, we have for $p \le p^*$
$$-\lambda h + (1-\lambda h) \left(- h p + (1-h)
\widetilde v^{h,\lambda}_-(p) \right) = 
\widetilde v^{h,\lambda}_-(p) \iff 
\widetilde v^{h,\lambda}_-(p) = 
\frac{(h \lambda - 1) p - \lambda}{1+(1-h)\lambda}.$$

$p^*$ is the approximate solution of the equation
\begin{multline*}
-\lambda h + (1-\lambda h) \left(- h p + (1-h) 
\frac{(h \lambda - 1) p - \lambda}{1+(1-h)\lambda} \right) = \\
-\lambda h + (1-\lambda h) 
\left(\frac{(h \lambda - 1) \left(p - \frac{h p}{2}\right) - \lambda}
{2(1+(1-h)\lambda)} + 
\frac{(h \lambda - 1) (p + h - h p) - \lambda}{2(1+(1-h)\lambda)}\right),
\end{multline*}
from which it is easy to find $p^*$ and see that 
$$p^* =
\frac{4\lambda + 2 - 2 \lambda h}{4\lambda + 3 - 7 \lambda h}
\xrightarrow{h \to 0}
\frac{4\lambda + 2}{4\lambda + 3}.$$
So we take $p^* = \frac{4\lambda + 2}{4\lambda + 3}$.
For $p \ge p^*$, $\widetilde v^{h,\lambda}_-(p)$ 
is a solution of the equation (in $f(p)$)
$$f(p) = -\lambda h + (1- \lambda h) 
\left(\frac{1}{2} f\left(p- \frac{h p}{2}\right) + \frac{1}{2} 
f(p + h - h p)\right).$$
Assuming that $\widetilde v^{h,\lambda}_-(p)$ is differentiable if 
$p > p^*$, we have
$$\widetilde v^{h,\lambda}_-(p) = -\lambda h + (1 - \lambda h)
\left(\frac{1}{2} \left(\widetilde v^{h,\lambda}_-(p) - \frac{1}{2} h p \; 
\left(\widetilde v^{h,\lambda}_-\right)'(p) \right) + 
\frac{1}{2} \left(\widetilde v^{h,\lambda}_-(p) + 
(h - h p)\left(\widetilde v^{h,\lambda}_-\right)'(p)\right)\right) + o(h).$$
Thus we have for small $h$
$$\begin{cases*} 
\lambda \widetilde v^{h,\lambda}_-(p) \approx - \lambda
-\frac{1}{4} p \left(\widetilde v^{h,\lambda}_-\right)'(p) 
+ \frac{1}{2} (1 - p) \left(\widetilde v^{h,\lambda}_-\right)'(p), \quad
\text{if } p \in (p^*,1);\\
\widetilde v^{h,\lambda}_-(p^*) = 
\frac{- p^* - \lambda}{1+\lambda}.
\end{cases*}$$

By solving the differential equation, we obtain for small $h$
$$\widetilde v^{h,\lambda}_- (p) = -1 + C (3 p - 2)^{-4\lambda/3},$$
where $C \in \R$. 
Taking into account the boundary condition, we obtain for small $h$
$$\widetilde v^{h,\lambda}_- (p) = -1 + 
\frac{(4\lambda)^{4\lambda/3}}
{(1+\lambda) (3 + 4\lambda)^{1+(4\lambda/3)}} 
\left(3 p - 2\right)^{-4\lambda/3}.$$

Thus we have for small $h$
\begin{equation*}
\label{fffg55f2}
\widetilde v^{h,\lambda}_- (p) \approx w^-_\lambda (p).
\end{equation*}

\noindent \textbf{The value $\widetilde v^{h,\lambda}_+$ of the
$\pmb{\lambda}$-discounted POMDP $\pmb{\widetilde G_{h,\lambda}^+}$ (heuristic computation).} 
We denote by $p$ (respectively $1-p$) the probability that the initial state is $\omega_5$ (respectively $\omega_4$).

By Proposition~\ref{shapley_dop} we have
\begin{equation}
\widetilde v^{h,\lambda}_+(p) = 
\lambda h + (1-\lambda h)\min\bigg\{
\underbrace{h p }_
{\substack{\text{Player chooses } Q, \\ 
\text{the signal is } \alpha}}
+
\underbrace{(1-h) \widetilde v^{h,\lambda}_+(p)}_
{\substack{\text{Player chooses } Q, \\ 
\text{the signal is } \beta}}
;
\underbrace{\frac{1}{3} \widetilde v^{h,\lambda}_+ 
\left(p- \frac{2 h p}{3}\right)}_
{\substack{\text{Player chooses } C, \\ 
\text{the signal is } \alpha}}
+
\underbrace{\frac{2}{3} 
\widetilde v^{h,\lambda}_+(p + h - h p)}_
{\substack{\text{Player chooses } C, \\ 
\text{the signal is } \beta}} \bigg\}.
\label{eqE2}
\end{equation}

By performing computations analogous to those above, we have for small $h$
\begin{equation*}
\label{fffg55f1}
\widetilde v^{h,\lambda}_+ (p) \approx w^+_\lambda (p) \qedhere
\end{equation*}
\end{proof}

\subsubsection{Proof that $v_{0,\lambda}^+ = w_\lambda^+$ and $v_{0,\lambda}^- = w_\lambda^-$}
\label{rPART3}

We now present a formal proof of Lemma~\ref{MainLemma}. First, we give a simple lemma.

\begin{lemma}
\label{lemalema4}
Let $x, y, z \in \R$. In the one-shot matrix game
\begin{equation*}
\begin{tabular}{|c|c|c|}
\hline
& $L$ & $R$\\
\hline
$T$ & $x$ & $y$\\
\hline
$B$ & $y$ & $x$\\
\hline
$Q$ & $z$ & $z$\\
\hline
\end{tabular} 
\end{equation*}

\begin{enumerate}
  \item $\frac{1}{2} L + \frac{1}{2} R$ is an optimal strategy of Player~2;
  \item If $z \ge \frac{1}{2} x + \frac{1}{2} y$, then $Q$ is an optimal strategy of Player~1;
  \item If $z < \frac{1}{2} x + \frac{1}{2} y$, then $\frac{1}{2} T + \frac{1}{2} B$ is an optimal strategy of Player~1.
\end{enumerate}
\end{lemma}

\begin{proof}[Proof of Lemma~\ref{MainLemma}]

We prove assertion 1). First, note that $w_\lambda^-(p)$ is 
continuously differentiable.
If $p \neq \frac{4\lambda + 2}{4\lambda + 3}$, this is clear. To prove the 
differentiability at 
$p = \frac{4\lambda + 2}{4\lambda + 3}$, 
one may verify that
$$\left(-\frac{p + \lambda}{1+\lambda}\right)'
\left(\frac{4\lambda + 2}{4\lambda + 3}\right) = 
\left(-1 + 
\frac{ (4 \lambda)^{4\lambda/3}}
{(1+\lambda) (3 + 4\lambda)^{1+(4\lambda/3)}} 
\left(3 p - 2\right)^{-4\lambda/3}\right)'
\left(\frac{4\lambda + 2}{4\lambda + 3}\right) = 
-\frac{1}{1+\lambda}.$$

Denote $I = \{T, B, Q\}, J = \{L, R\}$.
Also, denote by $\overline p = (p_1,p_2,p_3,p_4)$ a probability distribution over the states of $G_{h,\lambda}^-$, where
$$p_i = 
\begin{cases}
\text{the probability that the initial state is } \omega_1, &\text{if }i=1;\\
\text{the probability that the initial state is } \omega_2, &\text{if }i=2;\\
\text{the probability that the initial state is } \omega_3^*, &\text{if }i=3;\\
\text{the probability that the initial state is } 0^*, &\text{if }i=4.
\end{cases}$$
Denote by $\overline v_{h,\lambda}(\overline p)$ the value of the 
$\lambda$-discounted game $G_{h,\lambda}^-$ with 
initial probability distribution $\overline p$. It is clear that
\begin{equation}
\label{dsn6j2k22}
\overline v_{h,\lambda}(\overline p) = 
(p_1+p_2) v_{h,\lambda}^-\left(\frac{p_2}{p_1+p_2}\right) - p_3.
\end{equation}

Since $G_{h,\lambda}^-$ is a state-blind game, by \cite[Theorem~2]{Nov24a} we know that $\overline v_{0,\lambda} (\overline p)$ is the unique viscosity solution of the partial differential equation
\begin{equation}
\label{mc3ez8}
\lambda v(\overline p) = -\lambda(p_1+p_2+p_3) + \texttt{Val}_{I\times J} 
\left[\langle \overline p * q(i,j), 
\nabla v(\overline p) \rangle \right];
\end{equation}

We are going to check that 
$$\overline w_\lambda(\overline p) := 
(p_1+p_2) w_\lambda^- \left(\frac{p_2}{p_1+p_2}\right) - p_3$$
is a classical solution of \eqref{mc3ez8}.
First, we need to compute the kernel $q$ for 
$G_{h,\lambda}^-$.
Using Tables~\ref{eenrn43}--\ref{eenrn41} (with the replacement of state $\omega_5$ by state $0^*$), we obtain
\begin{align*}
  &q(\omega_2 \mid i,j, \omega_2)=
  \begin{cases}
    0, &\text{if } (i,j)=(T,R) \text{ or } (B,L);\\
    -\frac{1}{2}, &\text{if } (i,j)=(T,L) \text{ or } (B,R);\\
    -1, &\text{if } (i,j)=(Q,L) \text{ or } (Q,R).
  \end{cases}
  &q(\omega_1 \mid i,j, \omega_2)=
  \begin{cases}
    \frac{1}{2}, &\text{if } (i,j)=(T,L) \text{ or } (B,R);\\
    0, &\text{otherwise.}
  \end{cases}\\
  &q(\omega_3^* \mid i,j, \omega_2)=
  \begin{cases}
    1, &\text{if } (i,j)=(Q,L) \text{ or } (Q,R);\\
    0, &\text{otherwise.}
  \end{cases}
  &q(0^* \mid i,j, \omega_2) = 0 \qquad \forall \; (i,j);\\
  &q(\omega_1 \mid i,j, \omega_1)=
  \begin{cases}
    0, &\text{if } (i,j)=(T,L) \text{ or } (B,R);\\
    -1, &\text{otherwise}.
  \end{cases}
  &q(\omega_2 \mid i,j, \omega_1)=
  \begin{cases}
    1, &\text{if } (i,j)=(T,R) \text{ or } (B,L);\\
    0, &\text{otherwise.}
  \end{cases}\\
  &q(\omega_3^* \mid i,j, \omega_1)= 0 \qquad \forall \; (i,j);
  &q(0^* \mid i,j, \omega_1)=
  \begin{cases}
    1, &\text{if } (i,j)=(Q,L) \text{ or } (Q,R);\\
    0, &\text{otherwise.}
  \end{cases}\\
  &q(s \mid i,j, \omega_3^*)= 0 \qquad \forall \; (i,j), \; s;
  &q(s \mid i,j, 0^*)= 0 \qquad \forall \; (i,j), \; s.
\end{align*}

Hence
$$\overline p * q(i,j) =
  \begin{cases}
    (\frac{1}{2} p_2,-\frac{1}{2} p_2,0,0), &
\text{if } (i,j)=(T,L) \text{ or } (B,R);\\
    (-p_1,p_1,0,0), &\text{if } (i,j)=(T,R) \text{ or } (B,L);\\
    (-p_1, -p_2, p_2, p_1), &\text{if } (i,j)=(Q,L) \text{ or } (Q,R).
  \end{cases}$$

Now, we need to find the value of the one-shot matrix game
\begin{equation}
\begin{tabular}{|c|c|c|}
\hline
& $L$ & $R$\\
\hline
$T$ & $\frac{p_2}{2} \left(
\frac{\partial \, \overline w_{\lambda}}{\partial p_1} - 
\frac{\partial \, \overline w_{\lambda}}{\partial p_2}
\right)$ & 
$- p_1 \left(
\frac{\partial \, \overline w_{\lambda}}{\partial p_1} - 
\frac{\partial \, \overline w_{\lambda}}{\partial p_2}
\right)$\\
\hline
$B$ & $- p_1 \left(
\frac{\partial \, \overline w_{\lambda}}{\partial p_1} - 
\frac{\partial \, \overline w_{\lambda}}{\partial p_2}
\right)$ & 
$\frac{p_2}{2} \left(
\frac{\partial \, \overline w_{\lambda}}{\partial p_1} - 
\frac{\partial \, \overline w_{\lambda}}{\partial p_2}
\right)$\\
\hline
$Q$ & $-p_1 \frac{\partial \, \overline w_{\lambda}}{\partial p_1} - 
p_2 \frac{\partial \, \overline w_{\lambda}}{\partial p_2} - p_2$ & 
$-p_1 \frac{\partial \, \overline w_{\lambda}}{\partial p_1} - 
p_2 \frac{\partial \, \overline w_{\lambda}}{\partial p_2} - p_2$\\
\hline
\end{tabular} 
\label{ds3MA77x}
\end{equation}

We consider two cases.

\textbf{Case 1:} $\frac{p_2}{p_1+p_2} < \frac{4\lambda + 2}{4\lambda + 3}$.
By \eqref{dsn6j2k22} we have
\begin{align*}
&\frac{\partial \, \overline w_{\lambda}}{\partial p_1} = 
-\frac{\lambda}{1+\lambda}; \qquad
\frac{\partial \, \overline w_{\lambda}}{\partial p_2} = -1; \qquad
\frac{\partial \, \overline w_{\lambda}}{\partial p_3} = -1; \qquad
\frac{\partial \, \overline w_{\lambda}}{\partial p_4} = 0.
\end{align*}

Hence, the game \eqref{ds3MA77x} may be rewritten as follows.
\begin{equation*}
	\begin{tabular}{|c|c|c|}
\hline
& $L$ & $R$\\
\hline
$T$ & $\frac{p_2}{2(1+\lambda)}$ & $\frac{-p_1}{1+\lambda}$\\
\hline
$B$ & $\frac{-p_1}{1+\lambda}$ & $\frac{p_2}{2(1+\lambda)}$\\
\hline
$Q$ & $\frac{\lambda p_1}{1+\lambda}$ & $\frac{\lambda p_1}{1+\lambda}$\\
\hline
	\end{tabular} 
\end{equation*}

We have
\begin{align*}
\left(\frac{1}{2} \cdot \frac{p_2}{2(1+\lambda)} + 
\frac{1}{2} \cdot \frac{-p_1}{1+\lambda}\right) -
\frac{\lambda p_1}{1+\lambda} &= 
\frac{p_2 - (4\lambda +2)p_1}{4(1+\lambda)} = 
\frac{(4\lambda +3)p_2 - (4\lambda +2)(p_1+p_2)}{4(1+\lambda)} \\ 
&= \frac{(4\lambda +3)(p_1+p_2)\left( \frac{p_2}{p_1+p_2} - 
\frac{4\lambda +2}{4\lambda+3}\right)}
{4(1+\lambda)} \le 0.
\end{align*}

Hence, by Lemma~\ref{lemalema4}, $Q$ is an optimal strategy for Player~1. Now, it remains to verify that the partial differential equation \eqref{mc3ez8} holds. We have
\begin{align*}
&\lambda \overline w_{\lambda}(\overline p) + 
\lambda(p_1+p_2+p_3) - \texttt{Val}_{I\times J} 
\left[\langle \overline p * q(i,j), 
\nabla \overline w_{\lambda}(\overline p) \rangle \right] \\
=& 
\left(- \lambda (p_1+p_2) 
\left(\frac{\frac{p_2}{p_1+p_2}+\lambda}{1+\lambda}\right) -
\lambda p_3\right) + \lambda(p_1+p_2+p_3) - \frac{\lambda p_1}{1+\lambda} \\
=&
\frac{- \lambda p_2 - \lambda^2 (p_1+p_2)}{1+\lambda} +
\frac{\lambda (p_1+p_2)(1+\lambda)}{1+\lambda} -
\frac{\lambda p_1}{1+\lambda}=0.
\end{align*}

\textbf{Case 2:} $\frac{p_2}{p_1+p_2} \ge \frac{4\lambda + 2}{4\lambda + 3}$.
In this case, we have
\begin{align*}
&\hspace{-0.4cm}\frac{\partial \, \overline w_{\lambda}}{\partial p_1} = 
- 1 + 
\frac{1}{(4\lambda + 3)(\lambda + 1)}\left(\frac{ 4\lambda (p_1+p_2)}
{(4\lambda+3)(p_2 - 2 p_1)}\right)^{4\lambda/3} + \:
\frac{p_2 (p_1+p_2)^{4\lambda/3}}{\lambda + 1}\left(\frac{ 4\lambda}
{(4\lambda+3)(p_2 - 2 p_1)}\right)^{1+(4\lambda/3)};     \\
&\hspace{-0.4cm}\frac{\partial \, \overline w_{\lambda}}{\partial p_2} =
- 1 + 
\frac{1}{(4\lambda + 3)(\lambda + 1)}\left(\frac{ 4\lambda (p_1+p_2)}
{(4\lambda+3)(p_2 - 2 p_1)}\right)^{4\lambda/3} - \:
\frac{p_1 (p_1+p_2)^{4\lambda/3}}{\lambda + 1}\left(\frac{ 4\lambda}
{(4\lambda+3)(p_2 - 2 p_1)}\right)^{1+(4\lambda/3)}.
\end{align*}

Hence
\begin{align*}
&\frac{\partial \, \overline w_{\lambda}}{\partial p_1} - 
\frac{\partial \, \overline w_{\lambda}}{\partial p_2} = 
\frac{1}{1+\lambda}
\left(\frac{ 4\lambda (p_1+p_2)}
{(4\lambda+3)(p_2 - 2 p_1)}\right)^{1+(4\lambda/3)};\\
&-p_1 \frac{\partial \, \overline w_{\lambda}}{\partial p_1} 
-p_2 \frac{\partial \, \overline w_{\lambda}}{\partial p_2} - p_2 
= p_1 - \frac{p_1+p_2}{(1+\lambda)(3+4\lambda)}
\left(\frac{ 4\lambda (p_1+p_2)}
{(4\lambda+3)(p_2 - 2 p_1)}\right)^{4\lambda/3}.
\end{align*}

A direct calculation shows that
\begin{align*}
&\left(\frac{1}{2} \cdot \frac{p_2}{2}
\left(\frac{\partial \, \overline w_{\lambda}}{\partial p_1} - 
\frac{\partial \, \overline w_{\lambda}}{\partial p_2}\right) + 
\frac{1}{2} \cdot (-p_1) 
\left(\frac{\partial \, \overline w_{\lambda}}{\partial p_1} - 
\frac{\partial \, \overline w_{\lambda}}{\partial p_2}\right)\right) + 
\left(p_1 \frac{\partial \, \overline w_{\lambda}}{\partial p_1} 
+ p_2 \frac{\partial \, \overline w_{\lambda}}{\partial p_2} + 
p_2\right) \\
&= 
(p_1 + p_2) \left(-\frac{p_1}{p_1+p_2} + \frac{1}{4\lambda+3}
\left(\frac{4\lambda}
{(4\lambda+3)\left(1-\frac{3 p_1}{p_1+p_2}\right)}\right)^{4\lambda/3}\right).
\end{align*}

Since $\frac{p_2}{p_1+p_2} \ge \frac{4\lambda + 2}{4\lambda + 3}$, 
it follows that 
$\frac{p_1}{p_1+p_2} \le \frac{1}{4\lambda + 3} < \frac{1}{3}.$
Consider the function
$f_\lambda : [0,1/3) \to \R$, defined by 
$$f_\lambda(p) = -p + \frac{1}{4\lambda+3}
\left(\frac{4\lambda}
{(4\lambda+3)\left(1 - 3 p\right)}\right)^{4\lambda/3}.$$

We have 
$$f_\lambda'(p) = -1 + 
\left(\frac{4\lambda}
{(4\lambda+3)\left(1 - 3 p\right)}\right)^{1 + (4\lambda/3)}
\text{, and }
f_\lambda'(p) = 0 \iff p = \frac{1}{4\lambda + 3}.$$

Note that $f_\lambda'(p)$ is a strictly increasing function, with 
$f_\lambda'(0) < 0$ and $f_\lambda'(p) \to +\infty$ as $p \to 1/3.$ Thus
for any $p \in [0,1/3)$ we have 
$$f_\lambda(p) \ge f_\lambda\left(\frac{1}{4\lambda + 3}\right) = 0.$$

Hence, by Lemma~\ref{lemalema4}, $\frac{1}{2}T + \frac{1}{2}B$ is an optimal strategy of Player~1. Now, it remains to verify that the partial differential equation \eqref{mc3ez8} holds. We have
\begin{align*}
&\lambda \overline w_{\lambda}(\overline p) + 
\lambda(p_1+p_2+p_3) - \texttt{Val}_{I\times J} 
\left[\langle \overline p * q(i,j), 
\nabla \overline w_\lambda(\overline p) \rangle \right] \\
=& 
\left(- \lambda (p_1+p_2) 
\left(-1 + \frac{ (4 \lambda)^{4\lambda/3}}
{(1+\lambda) (3 + 4\lambda)^{1+(4\lambda/3)}} 
\left(\frac{3 p_2}{p_1+p_2} - 2\right)^{-4\lambda/3}\right) -
\lambda p_3\right) + \lambda(p_1+p_2+p_3) \\
-& \left(\frac{p_2}{4} - \frac{p_1}{2}\right) \frac{1}{1+\lambda}
\left(\frac{ 4\lambda (p_1+p_2)}
{(4\lambda+3)(p_2 - 2 p_1)}\right)^{1+(4\lambda/3)} = 0.
\end{align*}

This concludes the proof of assertion~1. 

The proof of assertion~2 is analogous; we will therefore provide only a general outline.

Denote
$I = \{T, M, B\}, J = \{L, M, R, Q\}$.
Also, denote by $\overline p = (p_1,p_2,p_3,p_4)$ a probability distribution over the states of $G_{h,\lambda}^+$, where
$$p_i = 
\begin{cases}
\text{the probability that the initial state is } \omega_4, &\text{if }i=1;\\
\text{the probability that the initial state is } \omega_5, &\text{if }i=2;\\
\text{the probability that the initial state is } \omega_6^*, &\text{if }i=3;\\
\text{the probability that the initial state is } 0^*, &\text{if }i=4.
\end{cases}$$

Consider the equation (in $v(\overline p)$)
\begin{equation}
\lambda v(\overline p) = -\lambda(p_1+p_2+p_3) + \texttt{Val}_{I\times J} 
\left[A\right],
\label{65F32e}
\end{equation}
where $A$ is the one-shot matrix game
\begin{equation*}
\begin{tabular}{|c|c|c|c|c|}
\hline
& $L$ & $M$ & $R$ & $Q$ \\
\hline
$T$ & $\frac{2 p_2}{3} \left(
\frac{\partial v}{\partial p_1} - 
\frac{\partial v}{\partial p_2}
\right)$ & 
$- p_1 \left(
\frac{\partial v}{\partial p_1} - 
\frac{\partial v}{\partial p_2}
\right)$ &
$- p_1 \left(
\frac{\partial v}{\partial p_1} - 
\frac{\partial v}{\partial p_2}
\right)$&
$-p_1 \frac{\partial v}{\partial p_1} - 
p_2 \frac{\partial v}{\partial p_2} + p_2$\\
\hline
$M$ & $- p_1 \left(
\frac{\partial v}{\partial p_1} - 
\frac{\partial v}{\partial p_2}
\right)$ & 
$\frac{2 p_2}{3} \left(
\frac{\partial v}{\partial p_1} - 
\frac{\partial v}{\partial p_2}
\right)$ &
$- p_1 \left(
\frac{\partial v}{\partial p_1} - 
\frac{\partial v}{\partial p_2}
\right)$&
$-p_1 \frac{\partial v}{\partial p_1} - 
p_2 \frac{\partial v}{\partial p_2} + p_2$\\
\hline
$B$ & $- p_1 \left(
\frac{\partial v}{\partial p_1} - 
\frac{\partial v}{\partial p_2}
\right)$ & 
$- p_1 \left(
\frac{\partial v}{\partial p_1} - 
\frac{\partial v}{\partial p_2}
\right)$ &
$\frac{2 p_2}{3} \left(
\frac{\partial v}{\partial p_1} - 
\frac{\partial v}{\partial p_2}
\right)$&
$-p_1 \frac{\partial v}{\partial p_1} - 
p_2 \frac{\partial v}{\partial p_2} + p_2$\\
\hline
\end{tabular} 
\end{equation*}

Now, it remains to consider two cases 
$\bigl(\frac{p_2}{p_1+p_2} < \frac{9\lambda + 6}{9\lambda + 8}$ and
$\frac{p_2}{p_1+p_2} \ge \frac{9\lambda + 6}{9\lambda + 8}\bigr)$ to verify that 
$$\overline w_\lambda(\overline p) := 
(p_1+p_2) w_\lambda^+ \left(\frac{p_2}{p_1+p_2}\right) + p_3$$
is a classical solution of \eqref{65F32e}.
\end{proof}


Finally, we are ready to prove Theorem~\ref{thh2}. 

\subsubsection{Combining two ``half-games'' 
into a single one}
\label{rPART4}

\begin{proof}[Proof of Theorem~\ref{thh2}]
Recall that we denote $p = (p_1, p_2, p_4, p_5)$, where $p_i$ is the probability that the initial state is $\omega_i$. We also denote $p_- = (0,1,0,0)$ and $p_+ = (0,0,0,1)$.

By Lemma~\ref{dwjbckdsjb}, we have
$$
\begin{cases*}
  v_{h, \lambda} (p_+) = v_{h, \lambda}^+
  (v_{h, \lambda} (p_-),1) = 
  (1-v_{h, \lambda} (p_-)) v_{h, \lambda}^+(1) + 
  v_{h, \lambda} (p_-); \\
  v_{h, \lambda} (p_-) = v_{h, \lambda}^-
  (v_{h, \lambda} (p_+),1) = (1+v_{h, \lambda} (p_+))
  v_{h, \lambda}^-(1) + 
  v_{h, \lambda} (p_+).
\end{cases*}
$$

By solving the system for the variables $v_{h, \lambda} (p_+)$ and $v_{h, \lambda} (p_-)$, we obtain
\begin{equation}
\begin{cases*}
  v_{h, \lambda} (p_+) = 
  - \frac{v_{h, \lambda}^-(1) + v_{h, \lambda}^+(1) - 
  v_{h, \lambda}^-(1) v_{h, \lambda}^+(1)}
  {v_{h, \lambda}^-(1) - v_{h, \lambda}^+(1) - 
  v_{h, \lambda}^-(1) v_{h, \lambda}^+(1)}; \\
  v_{h, \lambda} (p_-) = 
  - \frac{v_{h, \lambda}^-(1) + v_{h, \lambda}^+(1) + 
  v_{h, \lambda}^-(1) v_{h, \lambda}^+(1)}
  {v_{h, \lambda}^-(1) - v_{h, \lambda}^+(1) - 
  v_{h, \lambda}^-(1) v_{h, \lambda}^+(1)}.
\end{cases*}
\label{qds42}
\end{equation}

By Lemma~\ref{MainLemma}, we have 
\begin{align*}
&v_{0, \lambda}^-(1) = -1 +
\frac{(4 \lambda)^{4\lambda/3}}
{(1+\lambda) (3 + 4\lambda)^{1+(4\lambda/3)}} 
\left(3 \cdot 1 - 2\right)^{-4\lambda/3};\\
&v_{0, \lambda}^+(1) = 1 - 
\frac{2 (18 \lambda)^{9\lambda/8}}
{(1+\lambda) (8 + 9\lambda)^{1+(9\lambda/8)}} 
\left(8 \cdot 1 - 6\right)^{-9\lambda/8}.
\end{align*}

Hence 
\begin{align}
\label{qsc5fg}
\begin{split}
&\lim_{\lambda \to 0} 
v_{0, \lambda}^-(1) = -1 + \lim_{\lambda \to 0}
\frac{(4 \lambda)^{4\lambda/3}}
{(1+\lambda) (3 + 4\lambda)^{1+(4\lambda/3)}} = 
- \frac{2}{3};\\
&\lim_{\lambda \to 0} 
v_{0, \lambda}^+(1) = 1 - \lim_{\lambda \to 0}
\frac{2 (18 \lambda)^{9\lambda/8} \cdot 2^{-9\lambda/8} }
{(1+\lambda) (8 + 9\lambda)^{1+(9\lambda/8)}} 
= \frac{3}{4}.
\end{split}
\end{align}

By combining \eqref{qds42} and \eqref{qsc5fg}, we obtain
$$
\lim_{\lambda \to 0} v_{0,\lambda} (p_+) = 7/11 
\quad \text{ and } \quad
\lim_{\lambda \to 0} v_{0,\lambda} (p_-) = -5/11.
$$

If $p \in [0,1]$ and the initial state is $p(+) = (0, 0, 1-p, p)$, then by Lemmas~\ref{dwjbckdsjb} and \ref{MainLemma}, we have
\begin{align*}
\lim_{\lambda \to 0} v_{0,\lambda} (p(+)) =
\lim_{\lambda \to 0} v_{0,\lambda}^+\left(-\frac{5}{11},p\right)=
\frac{16}{11} \lim_{\lambda \to 0} 
v_{0,\lambda}^+(p) - \frac{5}{11}
=
\begin{cases}
  \frac{16}{11} \cdot \frac{3}{4} - \frac{5}{11} = \frac{7}{11},
  &\text{if } p \ge 3/4; \\
  \frac{16}{11} p - \frac{5}{11}, 
  &\text{if } p < 3/4.
\end{cases}
\end{align*}

If $p\in [0,1]$ and the initial probability distribution over the states is $p(-) = (1-p, p, 0, 0)$, then by Lemmas~\ref{dwjbckdsjb} and 
\ref{MainLemma}, we have
\begin{align*}
\lim_{\lambda \to 0} v_{0,\lambda}(p(-)) =
\lim_{\lambda \to 0} 
v_{0,\lambda}^-\left(\frac{7}{11},p\right) =
\frac{18}{11} \lim_{\lambda \to 0} 
v_{0,\lambda}^-(p) + \frac{7}{11} = 
\begin{cases}
  - \frac{18}{11} \cdot \frac{2}{3} + \frac{7}{11} = -\frac{5}{11},
  &\text{if } p \ge 2/3; \\
  - \frac{18}{11} p + \frac{7}{11}, 
  &\text{if } p < 2/3.
\end{cases}
\end{align*}

Finally, for any initial probability distribution $p = (p_1, p_2, p_4, p_5)$ over the states, we have
$$\lim_{\lambda \to 0} v_{0,\lambda}(p) =
\begin{cases*}
  (p_1+p_2) \lim_{\lambda \to 0} v_{0,\lambda}(\overline p(-))
   + (p_4+p_5) \lim_{\lambda \to 0} v_{0,\lambda}(\overline p(+))
  \hfill,
  \text{if } p_1 + p_2 > 0 \text{ and } p_3 + p_4 > 0; \\
  (p_1+p_2) \lim_{\lambda \to 0} v_{0,\lambda}(\overline p(-)), \hfill
  \text{if } p_4 + p_5 = 0; \\
  (p_4+p_5) \lim_{\lambda \to 0} v_{0,\lambda}(\overline p(+)), \hfill
  \text{if } p_1 + p_2 = 0.
\end{cases*},
$$
where
$$\overline p(-) = 
\left(\frac{p_1}{p_1+p_2}, \frac{p_2}{p_1+p_2}, 0, 0\right) 
\quad \text{ and } \quad
\overline p(+) = \left(0, 0, \frac{p_4}{p_4+p_5}, \frac{p_5}{p_4+p_5}\right).$$
This proves the existence of $\lim_{\lambda \to 0} v_{0,\lambda}$.
\end{proof}

\subsubsection{A discussion of the proof of Theorem~\ref{thh2}}
\label{rPART5}

Here, we give several remarks discussing the above proof.


\begin{remark}
  A consequence of the proof is that the limit $\lim\nolimits_{\lambda \to 0} v_{0,\lambda}(p)$ is uniform in $p$. In contrast, even the pointwise limit $\lim\nolimits_{\lambda \to 0} v_{1,\lambda}(p)$ does not exist. \demo
\end{remark}

\begin{remark}
  From \S\ref{rPART4}, we obtain that
\[\lim\limits_{\lambda \to 0} v_{0,\lambda} (p) =
\begin{cases}
  -\frac{5}{11} p_1 - \frac{5}{11} p_2 + 
  \frac{7}{11}p_4 + \frac{7}{11} p_5, &
  \text{if } \frac{p_2}{p_1+p_2} \ge \frac{2}{3} \text{ or } p_1 = p_2 = 0,
  \text{ and }\\ 
  &\phantom{if } 
  \frac{p_5}{p_4+p_5} \ge \frac{3}{4} \text{ or } p_4 = p_5 = 0; \\
  \frac{7}{11} p_1 - p_2 + 
  \frac{7}{11} p_4 + \frac{7}{11} p_5, &
  \text{if } \frac{p_2}{p_1+p_2} < \frac{2}{3} \text{ or } p_1 = p_2 = 0,
  \text{ and }\\ 
  &\phantom{if } 
  \frac{p_5}{p_4+p_5} \ge \frac{3}{4} \text{ or } p_4 = p_5 = 0; \\
  -\frac{5}{11} p_1 - \frac{5}{11} p_2 - 
  \frac{5}{11} p_4 + p_5, &
  \text{if } \frac{p_2}{p_1+p_2} \ge \frac{2}{3} \text{ or } p_1 = p_2 = 0,
  \text{ and }\\ 
  &\phantom{if } 
  \frac{p_5}{p_4+p_5} < \frac{3}{4} \text{ or } p_4 = p_5 = 0; \\
  \frac{7}{11} p_1 - p_2 - \frac{5}{11} p_4 + p_5,
  &
  \text{if } \frac{p_2}{p_1+p_2} < \frac{2}{3} \text{ or } p_1 = p_2 = 0,
  \text{ and }\\ 
  &\phantom{if } 
  \frac{p_5}{p_4 + p_5} < \frac{3}{4} \text{ or } p_4 = p_5 = 0.
\end{cases}\]\demo
\end{remark}

\begin{remark}
In this paper, we assumed that the stage duration does not depend on the stage number. In fact, the statement of Theorem~\ref{thh2} remains valid in a more general setting where the stage duration may depend on the stage number. This more general version of Theorem~\ref{thh2} is formulated and proven in the author's PhD thesis; see \cite{Nov25a}.\demo
\end{remark}

\begin{remark}
In this (somewhat informal) remark, we try to explain why $G^\lambda_1$ has no asymptotic value in discrete time (when each stage has duration $h = 1$), but has an asymptotic value in continuous time (when the stage duration $h$ approaches $0$).

First, we consider the discrete case: by waiting sufficiently long, Player~1 (respectively Player~2) can ensure that the probability $\hat p$ of being in state $\omega_1$ (respectively $\omega_4$) is arbitrarily close to $1$. The expected number of stages required to reach any fixed $\hat p$ is bounded, so the payoff during these stages becomes negligible as $\lambda \to 0$. Due to this rapid alternation between the players, an oscillation arises as $\lambda$ tends to $0$. For more details, see \cite{Zil16} and \cite{RenZil20}; see also \cite{SorVig15}.

In contrast, in continuous time it is impossible for Player~1 (respectively Player~2) to reach $\omega_1$ (respectively $\omega_4$) with probability greater than $2/3$ (respectively $3/4$). To see this, we consider the POMDPs $\widetilde G_h^-$ and $\widetilde G_h^+$ introduced in \S\ref{rPART3}. 

We first consider $\widetilde G_h^-$, where the initial state is $\omega_2$ with probability $p$ and $\omega_1$ with probability $1-p$. The player first plays $C$ to quickly reduce his belief $\widetilde p$ that the current state is $\omega_2$, and then switches to $Q$. Remark~\ref{irenNeman2} states that after playing $C$, the player's updated belief $\widetilde p$ that the current state is $\omega_2$ is
$$E [\widetilde p] = \frac{1}{2} (p + h - h p) + \frac{1}{2} \left( p - \frac{h p}{2}\right) =p + \frac{h}{4} (2 - 3 p).$$
Thus, after playing $C$, the player's belief $\widetilde p$ moves closer to the critical value $2/3$. 

As $h \to 0$, the belief $\widetilde p$ evolves more smoothly and the drift near $p=2/3$ vanishes, making it increasingly difficult for $\widetilde p$ to cross that threshold. Hence, as $h\to 0$, it is never profitable for the player to choose $C$ when $p \le 2/3$. If $p > 2/3$, then the player may play $C$ for some time to reduce $\widetilde p$. As $\lambda$ becomes smaller, he can wait longer, allowing $\widetilde p$ to approach $2/3$ more closely.

The analysis of $\widetilde G_h^+$ is analogous. Here, the initial state is $\omega_5$ with probability $p$ and $\omega_4$ with probability $1-p$. Remark~\ref{irenNeman2} states that after playing $C$, the player's updated belief $\widetilde p$ that the current state is $\omega_5$ is
$$E [\widetilde p] = \frac{2}{3} (p + h - h p) + \frac{1}{3} \left(p - \frac{2 h p}{3}\right) = p + \frac{h}{9} (6 - 8 p).$$

Thus, after playing $C$, the player's belief $\widetilde p$ moves closer to the critical value $3/4$ (see Figure~\ref{belief_a}).

As in the previous case, when $h \to 0$ the drift near $p = 3/4$ vanishes, so choosing $C$ is unprofitable for $p \le 3/4$; if $p > 3/4$, the player may still play $C$ to reduce $\widetilde p$ and let it approach $3/4$ more closely. 

Thus the game is quickly ``absorbed'', and, as a result, no oscillation occurs as $\lambda$ tends to $0$. \demo

\end{remark}

\begin{figure}[h!]
\centering
\captionsetup[subfigure]{justification=centering}

\begin{subfigure}[t]{0.48\textwidth}
\centering
\begin{tikzpicture}[x=7.5cm, y=1cm, >=Stealth, thick]

  \draw[line width=1pt] (0,0) -- (1,0);

  \foreach \x/\lab in {0/0, 0.6666667/{\tfrac{2}{3}}, 1/1} {
    \draw[line width=1.2pt] (\x,0.12) -- (\x,-0.12);
    \filldraw[black] (\x,0) circle (3pt); 
    \node[below=7pt, font=\large] at (\x,0) {$p=\lab$};}

  \draw[->, line width=1pt] (0.12,0.4) -- (0.60,0.4);
  \draw[->, line width=1pt] (0.94,0.4) -- (0.73,0.4);

\end{tikzpicture}
\caption{$\widetilde G_h^-$: drift of $\widetilde p$ toward $\tfrac{2}{3}$.}
\label{belief_a}
\end{subfigure}\hfill
\begin{subfigure}[t]{0.48\textwidth}
\centering
\begin{tikzpicture}[x=7.5cm, y=1cm, >=Stealth, thick]

  \draw[line width=1pt] (0,0) -- (1,0);

  \foreach \x/\lab in {0/0, 0.75/{\tfrac{3}{4}}, 1/1} {
    \draw[line width=1.2pt] (\x,0.12) -- (\x,-0.12);
    \filldraw[black] (\x,0) circle (3pt); 
    \node[below=7pt, font=\large] at (\x,0) {$p=\lab$};}

  \draw[->, line width=1pt] (0.12,0.4) -- (0.72,0.4);
  \draw[->, line width=1pt] (0.94,0.4) -- (0.78,0.4);

\end{tikzpicture}
\caption{$\widetilde G_h^+$: drift of $\widetilde p$ toward $\tfrac{3}{4}$.}
\label{belief_b}
\end{subfigure}

\caption{Belief dynamics after action $C$ in $\widetilde G_h^-$ and $\widetilde G_h^+$. Arrows indicate the expected drift of the posterior $\widetilde p$ toward the attracting point. Note that a similar phenomenon can be observed in the MDP considered in Example~\ref{exxxx1} (see Figure~\ref{belief_update}).}
\end{figure}

We conclude with a conjecture and several open questions.

  \begin{conjecture}
    \label{conj1}
    For the game $G_1$ from Theorem~\ref{thh2}, we have
    \begin{enumerate}
      \item For any $h \in (0,1]$, the limit 
      $\lim\limits_{\lambda \to 0} v_{h, \lambda}$ does not exist;
      \item We have 
      $\left|\limsup\limits_{\lambda \to 0} 
      v_{h, \lambda}(p) - 
      \liminf\limits_{\lambda \to 0} 
      v_{h, \lambda}(p)\right| \to 0$ as $h$ approaches $0$, uniformly in $p$.
    \end{enumerate}
  \end{conjecture}

\begin{question}
  Can we say that for any finite state-blind stochastic game, the limit $\lim\limits_{\lambda \to 0} v_{0,\lambda}$ exists?
\end{question}

\begin{question}
  Can we say that for any finite stochastic game with deterministic public signals, the existence of
  $\lim_{\lambda \to 0} v_{1,\lambda}$ implies the existence of $\lim_{\lambda \to 0} v_{0,\lambda}$?
\end{question}

\begin{appendices}
\section{Computation of $v_{0,\lambda}(p)$ for Example~\ref{exxxx1}}
\label{appenA}

\begin{proof}
Denote by $A_n$ and $B_n$ the events defined by
\begin{align*}
A_n &:= \{ \text{The current state is } S_1, \text{ after action } 
C \text{ has been played } n \text{ times} \};\\
B_n &:= \{ \text{The current state is } S_2, \text{ after action } 
C \text{ has been played } n \text{ times} \}.
\end{align*}

We consider several cases.

\noindent\textbf{Case 1:} $p_2 \ge p_1$.

From the equations
\begin{equation}
  \label{eqqq800}
  P(A_{n+1}) = (1-h) P(A_n) + h P(B_n) \;\text{ and }\; 
  P(B_{n+1}) = (1-h) P(B_n) + h P(A_n)
\end{equation}
we deduce that $P(A_{n+1}) < P(B_{n+1}) < P(A_n)$ if $P(A_n) > P(B_n)$, and $P(B_{n+1}) < P(A_{n+1}) < P(B_n)$ if $P(B_n) > P(A_n)$. Since $p_2 = P(B_0) \ge P(A_0) = p_1$, the above inequalities imply that $P(B_n)$ is maximal when $n=0$. Thus an optimal strategy is to always play $Q$. A direct computation shows that $v_{0,\lambda} (p) = \frac{p_2-p_1}{1+\lambda},$ where $p = (p_1,p_2,0,0) \text{ with } p_2 \ge p_1.$

\noindent \textbf{Case 2:} $p_1 > p_2$.

Assume that $P(A_n) > P(B_n)$ for some $n$ and $h \le 1/2$. In this case, by \eqref{eqqq800} we have
\begin{align*}
&P(A_n) > P(B_n) \iff (1-2h)P(A_n) > (1-2h)P(B_n) \iff \\
\iff &(1-h) P(A_n) + h P(B_n) > (1-h) P(B_n) + h P(A_n) \iff 
P(A_{n+1}) > P(B_{n+1}).
\end{align*}
Since $P(A_0) = p_1 > p_2 = P(B_0)$, 
this implies that $P(A_n) > P(B_n)$ for all $n \in \N^*$. Thus the only optimal strategy is to always play $C$, and we have $v_{0,\lambda}(p) = 0,$ where $p = (p_1,p_2,0,0)$ with $p_1 > p_2.$

From cases 1 and 2, it follows that
$$v_{0,\lambda}(p) = \frac{1}{1+\lambda} \max\{0, p_2-p_1\}. \qedhere$$
\end{proof}
\end{appendices}

\section{Acknowledgements}
The author is grateful to his research advisor Guillaume Vigeral for constant attention to this work. The author is also grateful to Sylvain Sorin, Yijun Wan and Bruno Ziliotto for useful discussions. The author is grateful to an anonymous reviewer for helpful comments. This work was done while the author was a PhD student at CEREMADE, Université Paris-Dauphine.

\bibliographystyle{apalike2}
\bibliography{ref}
  
 \end{document}